\documentclass {amsart}

\usepackage{graphicx,enumerate, stmaryrd, color, url}
\usepackage{enumitem}
\usepackage{amsmath}
\usepackage{amssymb}
\usepackage{amsthm}
\usepackage{amsfonts}
\usepackage[hidelinks]{hyperref}
\usepackage[all]{xypic}
\usepackage{tikz-cd}
\usepackage{tikz}
\usepackage{import}

\newtheorem{thm}{Theorem}[section]

\newtheorem{lma}[thm]{Lemma}
\newtheorem{prop}[thm]{Proposition}

\newtheorem{cor}[thm]{Corollary}
\theoremstyle{definition}
\newtheorem{defn}[thm]{Definition}
\newtheorem{rmk}[thm]{Remark}

\newtheorem{ex}[thm]{Example}

\newcommand{\D}{\underline \delta}

\newcommand{\md}{\underline{d}}

\newcommand{\hX}{\widehat{X}}

\newcommand{\hG}{\widehat{G}}

\newcommand{\cA}{\mathcal{A}}
\newcommand{\cX}{\mathcal{X}}

\def\cL{\mathcal{L}}

\newcommand{\trop}{\operatorname{trop}}

\newcommand{\Gr}{\operatorname{Gr}}
\newcommand{\Spec}{\operatorname{Spec}}

\newcommand{\val}{\operatorname{val}}

\newcommand{\cO}{\mathcal{O}}
\newcommand{\cF}{\mathcal{F}}

\newcommand{\Pic}{{\operatorname{{Pic}}}}
\newcommand{\Jac}{{\operatorname{{Jac}}}}
\newcommand{\Div}{{\operatorname{{Div}}}}
\newcommand{\supp}{{\operatorname{{supp}}}}

\newcommand{\Aut}{{\operatorname{{Aut}}}}

\newcommand{\Star}{{\operatorname{{Star}}}}

\newcommand{\an}{\mathrm{an}}

\newcommand{\QQ}{\mathbb{Q}}

\newcommand{\GG}{\mathbb{G}}
\newcommand{\RR}{\mathbb{R}}
\newcommand{\PP}{\mathbb{P}}
\newcommand{\ZZ}{\mathbb{Z}}

\newcommand{\ud}{\underline{d}}
\newcommand{\hud}{\widehat{\underline{d}}}

\newcommand{\Spf}{\operatorname{Spf}}

\title{Compactified Jacobians as Mumford models}

\author{Karl Christ, Sam Payne and Jifeng Shen}

\address[Christ]{$~^1$Department of Mathematics\\
	Ben-Gurion University of the Negev\\P.O.Box 653 \\Beersheba\\ 84105\\  Israel\\ and\\  $~^2$Institute of Algebraic Geometry\\
	Leibniz University Hannover\\ Welfengarten 1 \\30167 Ha\-no\-ver\\  Germany }\email{kchrist@math.uni-hannover.de}

\address[Payne]{Department of Mathematics \\ University of Texas at Austin \\ 2515 Speedway, RLM 8.100\\ Austin, TX 78712} \email{sampayne@utexas.edu}

\address[Shen]{} \email{jif.shen@gmail.com}

\begin{document}

	\begin{abstract} 
		We show that relative compactified Jacobians of one-parameter smoothings of a nodal curve of genus $g$ are Mumford models of the generic fiber. Each such model is given by an admissible polytopal decomposition of the skeleton of the Jacobian. We describe the decompositions corresponding to compactified Jacobians explicitly in terms of the auxiliary stability data and find, in particular, that in degree $g$ there is a unique compactified Jacobian encoding slope stability, and it is induced by the tropical break divisor decomposition.
	\end{abstract} 
	
	\maketitle
	
	\setcounter{tocdepth}{1}
	\tableofcontents
	
	\numberwithin{equation}{subsection}

	\section{Introduction}
	
	Throughout, we work over a complete and algebraically closed rank 1 valued field $K$, with valuation ring $R$.
	Let $\cX$ be a one-parameter smoothing of a nodal curve, by which we mean a flat and proper scheme of relative dimension 1 over $R$ whose generic fiber $\cX_K$ is smooth of genus $g$ and whose special fiber $X$ is reduced with only nodal singularities. There are several well-known constructions of relative compactified Jacobians for $\cX$, i.e., schemes proper over $\Spec R$, with an action of $\Jac(\cX)$, whose generic fiber is $\Pic^d(\cX_K)$ and whose special fiber is a (fine or good) moduli space of sheaves on $X$ that satisfy a suitable stability condition \cite{Ishida78, Caporaso94, Simpson94, Esteves01}. Since these constructions assume that the base is locally Noetherian, we require in addition that $\cX$ is defined over the valuation ring in a discretely valued subfield of $K$.
	
	We focus here on the compactified Jacobians of nodal curves constructed by Oda and Seshadri \cite{OdaSeshadri79} and their extension to the relative setting by Ishida \cite{Ishida78}, where the stability condition for sheaves on the nodal curve $X$ is given by a numerical polarization $\phi$. The numerical polarization is a tuple of real numbers, one for each irreducible component of $X$, that sum to the degree $d$. The resulting construction includes all well-known compactifications for one-parameter smoothings of nodal curves. Indeed, while the constructions of compactified Jacobians by Caporaso \cite{Caporaso94}, Simpson \cite{Simpson94} and Esteves \cite{Esteves01} each allow for families of curves over more general base schemes, in the special case of one-parameter smoothings of nodal curves, the relative compactified Jacobians that they produce are a subset of those arising from Ishida's construction \cite[Remark 2.12]{MeloViviani12}. 
	In particular, for any ample line bundle $H$ on $X$ there is a numerical polarization $\phi^{H,d}$ that encodes slope stability with respect to $H$ (see Remark~\ref{rem:slope stability}).

	We denote by $\overline J_{\cX}(\phi)$ the relative compactified Jacobian corresponding to a numerical polarization $\phi$. The modular properties of the construction ensure that the special fiber $\overline J_{X}(\phi)$ depends only on the nodal curve $X$ and the numerical polarization $\phi$, and not the choice of smoothing.  Our main theorem says that each such compactified Jacobian $\overline J_{\cX}(\phi)$ is a Mumford model of $\Pic^d(\cX_K)$, constructed via nonarchimedean analytic uniformization from a polytopal decomposition of the skeleton of $\Pic^d(\cX_K)^\an$, as in \cite[\S4]{Gubler10}. 
	
	\begin{thm} \label{thm:Mumfordmodels}
		The compactified Jacobian $\overline J_{\cX}(\phi)$ is a Mumford model of $\Pic^d(\cX_K)$.
	\end{thm}
	
	\noindent By construction, any Mumford model is \'etale locally isomorphic to a toric variety over the valuation ring $R$, as defined by \cite{GublerSoto15}.  Hence, from the global structure given by Theorem~\ref{thm:Mumfordmodels}, we immediately recover information on the local structure of compactified Jacobians, cf. \cite{Casalaina-MartinKassViviani15}.
	
	\begin{cor}
		The inclusion of $\Pic^d(\cX_K)$ in the compactified Jacobian is toroidal, i.e., it is \'etale locally isomorphic to the inclusion of a split torus over $K$ into a toric variety over $\Spec R$.
	\end{cor}
	
	\begin{cor}
		The compactified Jacobian of a nodal curve is semi-toroidal, i.e., it is \'etale locally isomorphic to a torus invariant closed subvariety of a toric variety.
	\end{cor}

	\noindent Each Mumford model of $\Pic^d(\cX_K)$ is determined by a polytopal decomposition of the skeleton of its nonarchimedean analytification $\Pic^d(\cX_K)^\an$. In the proof of Theorem~\ref{thm:Mumfordmodels}, we explicitly specify the polytopal decomposition of the skeleton that gives rise to $\overline J_{\cX}(\phi)$.  The cells in this decomposition have a natural modular interpretation in tropical geometry, as follows. The \emph{type} of a rank 1 torsion free sheaf on the the nodal special fiber is its multidegree $\md$ and the set $S$ of nodes at which the sheaf is not locally free.  Then the cells in the skeleton correspond to the combinatorial types of $\phi$-polystable sheaves and parametrize the classes of tropical divisors given by $\md$ on the vertices of the dual graph, and with one additional point on each of the edges corresponding to nodes in $S$. See \S\ref{sec:stabledecomp} for further details.
 
	The decompositions obtained in this way are generalizations to metric graphs of the Namikawa decompositions for graphs studied by Oda and Seshadri \cite[\S 6]{OdaSeshadri79}. 
	In some cases, we identify these Namikawa decompositions with well-known decompositions that have been previously studied in tropical geometry.
	
	\begin{thm} \label{thm:break}
		The compactified Jacobian $\overline J_{\cX}(\phi^{H,g})$ in degree $g$ is the Mumford model associated to the break divisor decomposition of the skeleton of $\Pic^g(\cX_K)^\an$ for any numerical polarization $\phi^{H,g}$ induced by an ample line bundle $H$.
	\end{thm}

	The affine space of all numerical polarizations of degree $d$ comes with a natural wall and chamber structure; in the resulting polyhedral decomposition, two polarizations are in the relative interior of the same cell if and only if they induce the same stability condition on rank 1 torsion free sheaves \cite[\S 7]{OdaSeshadri79}. 
	 If a numerical polarization $\phi'$ is contained in the closure of the cell containing another polarization $\phi$, we say that $\phi$ specializes to $\phi'$. In this case, there is an induced surjective moduli map $\overline J_{\cX}(\phi) \to \overline J_{\cX}(\phi')$ (Corollary~\ref{cor:map compactified Jacobians}), and the Namikawa decomposition corresponding to $\phi$ is a refinement of the one corresponding to $\phi'$ (Corollary~\ref{cor:refinement}).  Since the moduli map $\overline J_{\cX}(\phi) \to \overline J_{\cX}(\phi')$ and the toroidal blowup corresponding to this refinement are both given by the identity on the general fiber, it follows that they are the same morphism.

	\begin{cor}
		Let $\phi$ be a numerical polarization specializing to $\phi'$.  Then the induced map between compactified Jacobians $\overline J_{\cX}(\phi) \to \overline J_{\cX}(\phi')$ is toroidal, i.e., it is \'etale locally isomorphic to a morphism of toric varieties over the valuation ring.  
	\end{cor}

	\begin{rmk}
		Instances of the main constructions in this paper have also appeared in earlier works on compactified Jacobians, with different language and  notation.  The tropicalization of the Jacobian and linear equivalence of divisors on the dual graph play central roles in \cite{OdaSeshadri79} and \cite{Caporaso94}, respectively. The connection to Mumford models was known previously in the special case where $\cX$ is stable, the polarization is given by the relative canonical divisor, and the degree is $g-1$.  In this situation, the compactified Jacobian is naturally identified with the stable pairs degeneration of $(\Pic^{g-1}(\cX_K), \Theta)$ \cite{Alexeev04}, which is the Mumford model associated to the polytopal decomposition of $\Pic^d(\Gamma)$ induced by the tropical theta divisor \cite{AlexeevNakamura99}.  See also \cite{MolchoWise18} and \cite{AbreuPacini20} for related constructions with logarithmic Picard groups and compactified universal Jacobians, respectively.
	\end{rmk}
	
	The paper is structured as follows: In \S\S\ref{sec:sheaves}-\ref{sec:stability} we briefly review torsion free sheaves on nodal curves and the construction of the compactified Jacobians. We follow the setup of Oda and Seshadri \cite{OdaSeshadri79} and Ishida \cite{Ishida78}; the relation to Simpson's construction is explained in Remark~\ref{rem:slope stability} and to the one of Esteves in Remark~\ref{rmk:v-quasistability}.  In \S\ref{sec:extensions} we discuss the relationship between the tropicalization of divisors on $\cX_K$ and extensions of line bundles to semistable models. In \S\ref{section:decompositions}, we study the Namikawa decompositions of the tropical Jacobian into $\phi$-polystable types, and explain the relation to the tropical theta divisor in degree $g-1$ and the break divisor decomposition in degree $g$. Finally, in \S\ref{sec:Mumford}, we prove Theorem~\ref{thm:Mumford}, relating Mumford models to compactified Jacobians, from which Theorems~\ref{thm:Mumfordmodels} and \ref{thm:break} follow immediately.		
	
	\medskip
	
	\noindent {\bf Acknowledgements.} 
	We benefited from many conversations with friends and colleagues related to this project, and are especially grateful to M. Baker, L. Caporaso, J. Kass, and J. Rabinoff. We thank the  referee for a careful reading and many thoughtful suggestions. KC was partially supported by the Israel Science Foundation (grant No. 821/16).  SP was partially supported by NSF DMS-1702428, NSF DMS-2001502, and NSF DMS-2053261

	\section{Line bundles and torsion free sheaves} \label{sec:sheaves}
	
	Throughout, we fix a projective curve $\cX$ over $\Spec R$ whose generic fiber $\cX_K$ is smooth of genus $g$, and whose special fiber $X$ is reduced and nodal. Since $K$ is complete, $R$ is Henselian and thus there is a section of $\cX \to \Spec R$ through every smooth point in the special fiber $X$.
	
	Let $\Pic^d(\cX)$ be the relative degree $d$ Picard scheme of $\cX$ over $\Spec R$. Its general fiber is $\Pic^d(\cX_K)$, which is a torsor over $\Pic^0(\cX_K)$ via tensor product.  Let $\Jac(\cX) \subset \Pic^0(\cX)$ be the generalized Jacobian, parametrizing degree $0$ line bundles with degree zero on each irreducible component of the special fiber $X$.  Our main purpose is to study the geometry of \emph{compactified Jacobians} of $\cX$ over $\Spec R$, i.e., flat and proper schemes over $\Spec R$ with an action of $\Jac(\cX)$, together with an equivariant identification of the general fiber with $\Pic^d(\cX_K)$.  The special fibers of the compactified Jacobians that we consider parametrize rank 1 torsion free sheaves on the special fiber $X$ that satisfy certain stability conditions.
	
	We follow the usual convention that a \emph{torsion free sheaf} on a reducible nodal curve is a coherent sheaf whose associated points are irreducible components. Such a torsion free sheaf has \emph{rank 1} if the pullback to the normalization of each component, modulo torsion, is a line bundle.  Any line bundle is a rank 1 torsion free sheaf.
	
	\subsection{Multidegrees}
	
	The set of points at which a rank 1 torsion free sheaf is not locally free is a subset of the nodes.  The stability conditions that we consider depend on two data: the set of nodes at which the torsion free sheaf is not locally free, and the \emph{multidegree}, i.e., the tuple of integers given by the degree of the pullback to the normalization of each component, modulo torsion.
	
	\subsection{Dual graphs}
	
	Let $G$ denote the dual graph of the special fiber $X$, with vertex set $V(G)$ corresponding to the irreducible components and edge set $E(G)$ corresponding to the nodes. Note that $G$ may have loops or multiple edges if some irreducible component is singular or if some pair of components intersects at multiple nodes. We write $X_v$ for the irreducible component of $X$ corresponding to a vertex $v$.  
	
	A \emph{divisor} on $G$ is an integer-valued function on the vertices. Let $\ud_v$ denote the value of a divisor $\ud$ on a vertex $v$.  The \emph{degree} of a divisor is then $$\deg(\ud) = \sum_{v\in V(G)} \ud_v.$$

	\subsection{Partial normalizations and blowups} \label{section:normalizations_and_blowups}
	
	We identify the edge set $E(G)$ with the set of nodes of $X$. Let $S \subset E(G)$ be a subset of the nodes.  We write $$\nu\colon X^\nu_S \to X$$ for the partial normalization of $X$ along $S$.  The dual graph of $X^\nu_S$ is $G - S$.  Let $\hX_S$ be the nodal curve obtained from $X^\nu_S$ by attaching a smooth $\PP^1$ through both points in the preimage of each node.  Then $\nu$ extends to a natural projection $$\pi\colon \hX_S \to X,$$ contracting each new component to the corresponding node.  
	
	If $X$ is embedded in a smooth surface, then $\pi$ is the restriction of the blowup of the surface along $S$ to the preimage of $X$.  Similarly, $\nu \colon X^\nu_S \to X$ is the restriction of this blowup to the strict transform of $X$. Based on this relation to blowups, the components contracted by $\pi$ are called \emph{exceptional curves}.  The dual graph $\hG_S$ of $\hX_S$ is obtained from $G$ by subdivision, adding one \emph{exceptional vertex} in the middle of each edge in $S$; the exceptional vertices in $\hG_S$ correspond to the exceptional curves in $\hX_S$.

	\medskip
	
	Torsion free sheaves on $X$ are naturally related to line bundles on partial normalizations, as follows.  Suppose $F$ is a rank $1$ torsion free sheaf on $X$. 
	
	\begin{defn}
		We say that a rank $1$ torsion free sheaf $F$ on $X$ is \emph{of type $(S, \ud)$} if $S$ is the set of nodes at which $F$ is not locally free and $\ud$ is the multidegree of $F$.
	\end{defn}
	
	In particular, if $F$ is a rank $1$ torsion free sheaf of type $(S, \ud)$ its degree is
	\begin{equation*}\label{degree}
		\deg(F) = \deg(\ud) + |S|.
	\end{equation*}
	
	\noindent Let $F$ be of type $(S, \ud)$, and let $\nu^{[*]} F$ denote the pullback of $F$ to $X^\nu_S$, modulo torsion.  Then $\nu^{[*]} F$ is a line bundle of multidegree $\ud$ on $X^\nu_S$, and $F \cong \nu_* \nu^{[*]} F.$  We also associate to $F$ a collection of line bundles on $\hX_S$, as follows.

	The relative projectivization $\PP_X(F)$ is a nodal curve isomorphic to $\hX_S$, and the tautological bundle on $\PP_X(F)$ has degree $1$ on exceptional components \cite[Proposition~5.5]{EstevesPacini16}.  We write $\hud_S$ for the multidegree of the tautological bundle on $\PP_X(F)$, so $\hud_S$ agrees with $\ud$ on the nonexceptional vertices in $\hG_S$, and has degree $1$ on the exceptional vertices.  Then we associate to $F$ the set of line bundles of multidegree $\hud_S$ on $\hX_S$ obtained by pulling back the tautological bundle under isomorphisms $\hX_S \cong \PP_X(F)$ over $X$.

	\begin{rmk}
		The isomorphism between $\hX_S$ and $\PP_X(F)$ over $X$ is not canonical, and the set of choices for this isomorphism is a torsor over $\Aut_X(\hX_S) \cong \GG_m^S$.  To see this, choose an identification of the two preimages of each node in $S$ with $0$ and $\infty$ and let the corresponding $\GG_m$ factor act by rescaling the coordinate of the corresponding exceptional component.  Note that $\Aut_X(\hX_S) \cong \GG_m^S$ also acts transitively on the set of isomorphism classes of line bundles of multidegree $\hud_S$ on $\hX_S$ whose restriction to $X^\nu_S$ is isomorphic to $\nu^{[*]} F$.
	\end{rmk}

	\begin{prop} \label{prop:bijections}
		The maps taking $F$ to $\nu^{[*]}F$ and to the collection of line bundles obtained by pulling back the tautological line bundle under isomorphisms $\hX_S\cong \mathbb{P}_X(F)$ over $X$ induce bijections between
		\begin{enumerate}
			\item isomorphism classes of rank 1 torsion free sheaves of type $(S,\ud)$ on $X$,
			\item isomorphism classes of line bundles of multidegree $\ud$ on $X^\nu_S$, and
			\item $\Aut_X(\hX_S)$-orbits of isomorphism classes of line bundles of multidegree $\hud_S$ on $\hX_S$.
		\end{enumerate}
		
	\end{prop}
	
	\begin{proof}
		The inverses take a line bundle $L$ on $X^\nu_S$ (respectively, $\hX_S$) to the torsion free sheaf $\nu_* L$ (respectively, $\pi_* L$). See also \cite[Lemmas 1.5 and 1.9]{Alexeev04}.
	\end{proof}
	
	\noindent The induced bijection between isomorphism classes of line bundles of multidegree $\ud$ on $X^\nu_S$, and $\Aut_X(\hX_S)$-orbits of isomorphism classes of line bundles of multidegree $\hud_S$ on $\hX_S$ is given by restricting a line bundle on $\hX_S$ to the subcurve $X^\nu_S$.
	
	\section{Compactified Jacobians} \label{sec:stability}
	
	In this section, we briefly recall the construction of the compactified Jacobians that we will consider. Over the central fiber $X$, the compactified Jacobian is constructed via a numerical polarization as in the work of Oda and Seshadri \cite{OdaSeshadri79}. Since $\cX$ admits enough sections, this construction naturally extends to the relative setting by work of Ishida \cite{Ishida78}.
	
	\subsection{Polarizations and semistability} 
	A compactified Jacobian over $\cX$ depends on the choice of a \emph{numerical polarization} $\phi$ on the central fiber $X$. That is, a collection of real numbers $\phi = (\phi_v)_{v \in V(G)}$, one for each irreducible component $X_v$ of $X$, such that the sum $\sum_{v \in V(G)} \phi_v$ is an integer $d$. We call $d$ the \emph{degree} of the polarization $\phi$.
	Then $\phi$ determines notions of stability and semistability of rank 1 torsion free sheaves of degree $d$ on $X$, as follows.
	
	For a subcurve $Y \subset X$ let $Y^c$ be the closure of $X \smallsetminus Y$, and $S^c = X^{\mathrm{sing}} \smallsetminus S$. For a rank 1 torsion free sheaf denote by $F_Y$ the maximal torsion free quotient of $F|_Y$. If $F$ has type $(S, \md)$, then the type of $F_Y$ is given by restricting $S$ and $\md$ to $Y$. In particular, we have
	\[
	\deg(F_Y) = \sum_{X_v \subset Y} \ud_v + |Y^{\mathrm{sing}} \cap S|.
	\]
	
	\begin{defn} \label{def:lower bound}
		Let $F$ be a rank 1 torsion free sheaf on $X$ of degree $d$ and type $(S, \ud)$. Then $F$ is called $\phi$-\emph{semistable} if the following inequality holds for all subcurves $Y \subset X$:
		\begin{equation} \label{eq:lower bound}
			\deg(F_Y) \geq \sum_{X_v \subset Y} \phi_v - \frac{|Y \cap Y^c|}{2}.
		\end{equation} 
		We call this inequality the \emph{basic lower bound}. If the inequality is strict for every proper subcurve $Y$, we call $F$ $\phi$-\emph{stable}.
	\end{defn}
	
	\noindent In particular, the $\phi$-semistability or $\phi$-stability of $F$ depends only on $S$ and $\ud$.  Every pair $(S, \ud)$ is the type of a rank 1 torsion free sheaf, and we say that $(S, \ud)$ is $\phi$-semistable or $\phi$-stable if the associated rank 1 torsion free sheaves are so. We say that a sheaf or its type are \emph{strictly $\phi$-semistable} along a proper subcurve $Y$ of $X$ if equality is achieved along $Y$ in the basic lower bound~\eqref{eq:lower bound}.
	
	\begin{rmk} \label{rmk:upper bound}
		Observe that \[\sum_{X_v \subset Y} \phi_v + \sum_{X_v \subset Y^c} \phi_v = d\] and \[\sum_{X_v \subset Y} \ud_v + \sum_{X_v \subset Y^c} \ud_v = d - |S|.\] It follows that the basic lower bound \eqref{eq:lower bound} is satisfied if and only if the inequality
		\begin{equation} \label{eq:bound_complement}
			\deg(F_{Y^c}) \leq \sum_{X_v \subset Y^c} \phi_v - |Y \cap Y^c \cap S|  + \frac{|Y \cap Y^c|}{2}
		\end{equation}
		is satisfied. Furthermore, equality in the basic lower bound \eqref{eq:lower bound} is achieved if and only if equality is achieved in inequality~\eqref{eq:bound_complement}.
		Since $\phi$-stability and $\phi$-semistability require the corresponding inequality to be satisfied on all proper subcurves, one can equivalently consider the \emph{basic upper bound}: 
		\begin{equation} \label{eq:upper bound}
			\deg(F_Y) \leq \sum_{X_v \subset Y} \phi_v - |Y \cap Y^c \cap S|  + \frac{|Y \cap Y^c|}{2}.
		\end{equation} 
Then $F$ is $\phi$-semistable if and only if \eqref{eq:upper bound} holds for all $Y$, and it is $\phi$-stable if and only if this inequality is strict for every proper subcurve.		
	\end{rmk}
	
	\begin{rmk}\label{rmk:shift polarization}
		Let $p$ be a smooth point of $X$ contained in an irreducible component $X_{v}$. Then $F$ is $\phi$-semi\-stable or $\phi$-stable if and only if $F(- p) $ is $(\phi -v)$-semistable or $(\phi -v)$-stable, where $(\phi -v)$ is the numerical polarization of degree $d - 1$ with value $\phi_{v'}$ on vertices different from $v$ and value $\phi_{v} - 1$ on $v$.
	\end{rmk}

	\begin{rmk} \label{rem:slope stability}
		Numerical polarizations were used by Oda and Seshadri in their foundational work on compactified Jacobians of nodal curves  
		\cite{OdaSeshadri79}. While they considered only $d = 0$, this readily translates to any degree, by Remark~\ref{rmk:shift polarization}. We follow the general presentation of numerical polarizations from \cite{MeloViviani12}. The literature on compactified Jacobians includes other encodings of equivalent polarization data; we refer to \cite[\S\S 1-2]{Alexeev04}, \cite[\S 2]{MeloViviani12} and \cite[\S 2]{Casalaina-MartinKassViviani15} for a detailed discussion. We note, in particular, that slope stability can be naturally encoded in terms of a numerical polarization, as follows.  Given an ample line bundle $H$ on $X$, let  
		\begin{equation}\label{eq:slope inequality}
			(\phi^{H,d})_v =  g(X_v) - 1 + \frac{\deg (H|_{X_v})}{\deg (H)} (d + 1 - g) + \frac{|X_v \cap X_v^c|}{2}.
		\end{equation}
Then the numerical polarization $\phi^{H,d}$ of degree $d$ precisely encodes slope stability with respect to $H$, as considered, e.g., in \cite{Simpson94}. Conversely, any numerical polarization $\phi$ with rational coefficients can be realized as such a $\phi^{H,d}$, up to a translation and degree shift as in Remark~\ref{rmk:shift polarization} \cite[Remark 2.9]{Casalaina-MartinKassViviani15}.
	\end{rmk}
	
	Let $S \subset E(G)$ be a subset of the nodes. We extend a numerical polarization $\phi$ on $X$ to a numerical polarization $\widehat{\phi}^S$ on $\hX_S$ by setting its value on exceptional components to zero. We define a numerical polarization $\phi^S$ of degree 
	$d - |S|$ on $X_S^\nu$ by setting 
	\[
	(\phi^S)_v = \phi_v - \frac{|X_v \cap S|}{2}. 
	\]
	Similarly, for a subcurve $Y \subset X$ such that the right hand side of the basic lower bound~\eqref{eq:lower bound}, $\sum_{X_v \subset Y} \phi_v - \frac{|Y \cap Y^c|}{2}$,  is an integer we define a numerical polarization $\phi^Y$ on $Y$ of degree $\sum_{X_v \subset Y} \phi_v - \frac{|Y \cap Y^c|}{2}$ by setting: 
	\[
	(\phi^Y)_v = \phi_v - \frac{|X_v \cap Y^c|}{2}. 
	\]
	
	\begin{rmk} \label{rem:nustable}
		Inserting the definition of $\phi^S$ in the basic lower bound \eqref{eq:lower bound} shows that a rank 1 torsion free sheaf $F$ on $X$ is $\phi$-semistable or $\phi$-stable if and only if $\nu^{[*]} F$ on $X_S^\nu$ is $\phi^S$-semistable or $\phi^S$-stable, respectively.
	\end{rmk}
	
	\begin{lma} \label{lma:connected}
		If a type $(S,\ud)$ is $\phi$-stable then $X_S^\nu$ is connected. 
	\end{lma}
	
	\begin{proof}
		Suppose $X_S^\nu$ is not connected. Then there is a proper subcurve $Y \subset X$ such that $Y \cap Y^c \subset S$. But then the basic lower bound \eqref{eq:lower bound} and the basic upper bound \eqref{eq:upper bound} on $\deg(F_Y)$ coincide, and thus strict inequality is not possible.
	\end{proof}
	
	\subsection{Polystability} \label{section:polystability}
	
	Let $F$ be $\phi$-semistable. A Jordan-H\"older filtration of $F$ is a collection $\{Y_1, \dots, Y_k\}$ of connected subcurves covering $X$ and a filtration
	\[
	0 = F_0 \subsetneq F_1 \subsetneq \dots \subsetneq F_k = F
	\]
	by subsheaves such that the quotient $F_{j}/F_{j-1}$ is a $\phi^{Y_j}$-stable sheaf on $Y_j$. We set 
	\begin{equation}
		\Gr(F) = F_1/F_0 \oplus F_2/F_1 \oplus \dots \oplus F_k/F_{k - 1}.
	\end{equation}
	While the filtration is not unique, the (unordered) tuple $\{Y_1, \dots, Y_k\}$ and the sheaf $\Gr(F)$ are unique.
	
	\begin{defn}\label{def:polystable}
		A $\phi$-semistable sheaf $F$ is $\phi$-\emph{polystable} if it is isomorphic to $\Gr(F)$.  
	\end{defn}
	
	\noindent To each $\phi$-semistable sheaf $F$, we associate the $\phi$-polystable sheaf $\Gr(F)$.	
	
	Note that a $\phi$-stable sheaf $F$ is $\phi$-polystable, since its Jordan-H\"older filtration is trivial, i.e., $F_0 = 0$ and $F_1 = F$. Thus we have the following implications:
	\[ 
	F \mbox{ is } \phi\text{-stable } \Rightarrow F \mbox{ is } \phi \text{-polystable } \Rightarrow F \mbox{ is } \phi\text{-semistable. }
	\]
	
	\begin{lma} \label{lma:polystable_char}
		Let $F$ be a $\phi$-semistable sheaf of type $(S, \ud)$ on $X$ and $\nu\colon X_S^\nu \to X$ the partial normalization. Then the following are equivalent:
		\begin{enumerate}
			\item \label{lma:polystable_char_1} $F$ is $\phi$-polystable,
			\item \label{lma:polystable_char_2} $\nu^{[*]} F$ is $\phi^S$-semistable and strictly  $\phi^S$-semistable only along unions of connected components of $X^\nu_{S}$.
		\end{enumerate}
	\end{lma}
	
	\begin{proof}		
		Note that we can write $F = F_1 \oplus F_2$ if and only if $F_1 = F_Y$ and $F_2 = F_{Y^c}$ for some proper subcurve $Y$ of $X$ with $|Y \cap Y^c| \subset S$. Note furthermore that using Remark~\ref{rem:nustable} and the additivity of the basic lower bound~\eqref{eq:lower bound} on connected components the assertion in (2) is equivalent to: $F$ is $\phi$-semistable -- which we already assume -- and if it is strictly $\phi$-semistable along a connected subcurve $Y$ of $X$, then $Y \cap Y^c \subset S$.
		
		Suppose $F$ is $\phi$-polystable. Then $F = \bigoplus_i F_{Y_i}$ with $F_{Y_i}$ $\phi^{Y_i}$-stable. Let $Y$ be a connected subcurve of $X$ along which $F$ is strictly $\phi$-semistable, i.e.  
		\begin{equation} \label{eq:strictly semistable}
		    \deg(F_Y) = \sum_{X_v \subset Y} \phi_v  - \frac{|Y \cap Y^c|}{2}. 
		\end{equation}
		Since $\bigoplus_i F_{Y_i}$ is not locally free at $Y_i \cap Y_i^c$, we have $Y_i \cap Y_i^c \subset S$ and thus it suffices to show that $Y$ is a union of the $Y_i$. Since the $Y_i$ cover $X$, there is a $Y_i$ such that $Y \cap Y_i \neq \emptyset$ and we write $Y = W \cup Z$ with $W = Y \cap Y_i$ and $Z = \overline { Y \setminus W}$. We have \[\deg(F_Y) = \deg(F_{W}) + \deg(F_{Z}) + |Z \cap W|\] since $W \cap Z \subset S$. Furthermore, we have \[\sum_{X_v \subset Y} \phi_v = \sum_{X_v \subset W} \phi_v + \sum_{X_v \subset Z} \phi_v\] and \[|Y \cap Y^c| = |W \cap W^c| + |Z \cap Z^c| - 2 |W \cap Z|.\] Inserting these identities in \eqref{eq:strictly semistable} and rearranging terms gives 
		\[
		\left(\deg(F_{W}) - \sum_{X_v \subset W} \phi_v + \frac{|W \cap W^c|}{2}\right) +  \left(\deg(F_{Z}) - \sum_{X_v \subset Z} \phi_v + \frac{|Z \cap Z^c|}{2}\right) = 0.
		\]
		Since $F$ is $\phi$-semistable, this implies that $\deg(F_{W}) - \sum_{X_v \subset W} \phi_v + \frac{|W \cap W^c|}{2} = 0$. Since $\phi_v = (\phi^{Y_i})_v + \frac{|X_v \cap Y_i^c|}{2}$ we obtain that $F_{Y_i}$ is strictly $\phi^{Y_i}$-semistable along $W$ on $Y_i$. Since $F_{Y_i}$ is by assumption $\phi^{Y_i}$-stable on $Y_i$, we get $W = Y_i$, as claimed.   
		
		For the converse, suppose $\nu^{[*]}F$ is $\phi^S$-semistable and strictly  $\phi^S$-semistable only along unions of connected components of $X^\nu_{S}$. Write $Y_i = \nu(C_i)$. Since $F \cong \nu_* \nu^{[*]}F$, we have $F = \bigoplus_i F_{Y_i}$. For any $W \subset Y_i$ along which  $F_{Y_i}$ is strictly $\phi^{Y_i}$-semistable on $Y_i$, the calculation above shows that $F$ is strictly $\phi$-semistable along $W$ on $X$. Thus by assumption $W = Y_i$ and $F_{Y_i}$ is $\phi^{Y_i}$-stable.
		Therefore, $F$ has a Jordan-H\"older filtration with successive quotients $F_{Y_i}$.  Finally, since $F$ is the direct sum of these successive quotients, it is $\phi$-polystable.
	\end{proof}
	
	\noindent By Lemma~\ref{lma:polystable_char}, the $\phi$-polystability of $F$ depends only on $S$ and $\ud$. We say that $(S, \ud)$ is $\phi$-polystable if the associated rank 1 torsion free sheaves are so. 
	
	\begin{prop} \label{prop:type_socle}
		Let $F$ be a $\phi$-semistable rank 1 torsion free sheaf. Then the type of $\Gr(F)$ depends only on the type of $F$. 
	\end{prop}
	\begin{proof}
		Using Remark~\ref{rem:nustable} and Lemma~\ref{lma:polystable_char}, we may assume $F$ is locally free by passing to $\nu^{[*]} F$ and restricting to connected components of $X_S^\nu$.  
		
		In this case, suppose $F'$ has the same type as $F$. Any Jordan-H\"older filtration of $F$ gives a Jordan-H\"older filtration of $F'$ by tensoring with $F' \otimes F^{-1}$ since stability depends only on the type of the sheaf. Since $F' \otimes F^{-1}$ has degree zero on each irreducible component, tensoring with $F' \otimes F^{-1}$ does not change the type of $\Gr(F)$ which proves the claim.
	\end{proof}
	
	\subsection{Compactified Jacobians}\label{sec:comp_jac}
	
	Let $\phi$ be a numerical polarization of degree $d$, which determines notions of $\phi$-semistability, $\phi$-stability, and $\phi$-polystability for rank 1 torsion free sheaves on $X$, as explained above. We next describe the resulting compactified Jacobians associated to $\phi$.
	
	\subsubsection{The special fiber}
	
	We write $\overline J_X(\phi)$ for the compactified Jacobian of $X$ associated to the numerical polarization $\phi$, as constructed by Oda and Seshadri \cite{OdaSeshadri79}. 
	
	Closed points in $\overline J_X(\phi)$ are in bijection with isomorphism classes of $\phi$-polystable rank 1 torsion free sheaves of degree $d$ on $X$. We write $[F]$ for the point in $\overline J_X(\phi)$ parametrizing the $\phi$-polystable sheaf $F$. The compactified Jacobian $\overline J_X(\phi)$ is reduced and projective. 
	Furthermore, $\overline J_X(\phi)$ is a good moduli space, in the sense of \cite{Alper13}, for the stack of $\phi$-semistable rank 1 torsion free sheaves of degree $d$ on $X$. The moduli map takes a $\phi$-semistable sheaf $F$ on $X$ to the point $[\Gr(F)]$ of $\overline J_X(\phi)$.
	
	In particular, we have a set theoretic decomposition by type
	\begin{equation} \label{stratification_simpson}
		\overline J_X(\phi) = \bigsqcup_{(S, \ud)} J_X^{(S, \ud)},
	\end{equation}
	where the union runs over $\phi$-polystable types $(S, \ud)$ such that $|S| + \deg(\ud) = d$; the stratum $J_X^{(S, \ud)}$ parametrizes $\phi$-polystable sheaves of type $(S, \ud)$ on $X$.
	
	\subsubsection{The relative setting}
	
Recall that we have fixed a one parameter smoothing $\mathcal X \to \Spec R$. The discussion above generalizes naturally to this relative setting, as follows. We say that a rank 1 torsion free sheaf on $\cX$ is $\phi$-semistable, $\phi$-stable or $\phi$-polystable if its restriction to the special fiber $X$ is so. The proposition below summarizes the results of \cite[\S 5]{Ishida78} that we will use in what follows.
	\begin{prop} \label{prop:relative comp jac}
		For any numerical polarization $\phi$ of degree $d$, there is a projective relative compactified Jacobian $\overline J_\cX(\phi) \to \Spec(R)$ with general fiber $\Pic^d(\cX_K)$ and special fiber $\overline J_X(\phi)$. It is a good moduli space for the stack of $\phi$-semistable torsion free sheaves of rank 1 and degree $d$ on $\mathcal X$.
	\end{prop}
	
	The $\Pic^0(\cX_K)$ action on $\Pic^d(\cX_K)$ extends to an action of the generalized Jacobian $\Jac(\cX)$ on $\overline J_\cX(\phi)$, as follows. A point of $\Jac(\cX)$ corresponds to a line bundle $\mathcal{L}$ on $\cX$ with degree zero on $\cX_K$ and on every irreducible component of $X$. It acts on $\overline J_\cX(\phi)$ by sending a point $[\mathcal{F}] \in \overline J_\cX(\phi)$ to $[\mathcal{F} \otimes \mathcal{L}]$. This action preserves the type $(S, \ud)$ of the restriction of $\mathcal{F}$ to the special fiber, and the action over the special fiber is transitive on the sheaves of each $\phi$-polystable type.
	
	\subsubsection{The wall and chamber decompostion on the space of polarizations}
	\label{sec:varying polarization}
	
	By definition, a numerical polarization of degree $d$ is an element of 
	\[
	H_d = \bigg \{\phi \in \RR^{|V(G)|}\, : \sum_{v \in V(G)} \phi_v = d \bigg \},
	\] and the compactified Jacobians we consider depend on $\phi$. By \cite[\S 7]{OdaSeshadri79} this dependence gives $H_d$ a wall and chamber structure which we will describe next (cf. also \cite[\S 3]{MeloRapagnettaViviani17}).
	
	\medskip
	
	\begin{defn} \label{def:general_polarization}
	    A numerical polarization $\phi$ is \emph{general} if $\sum_{X_v \subset Y} \phi_v - \frac{|Y \cap Y^c|}{2} \not \in \ZZ$ for every proper subcurve $Y$ of $X$.
	\end{defn} 
	\noindent If $\phi$ is general then the notions of $\phi$-semistability, $\phi$-stability and $\phi$-polystability coincide, because equality in the basic lower bound \eqref{eq:lower bound} can never be achieved. By \cite[Corollary 12.15]{OdaSeshadri79}, $\phi$ is general if and only if $\overline J_{X}(\phi)$ is a fine moduli space in the sense that it is isomorphic to the $\mathbb G_m$-rigidification of the stack of $\phi$-semistable torsion free sheaves of rank 1 and degree $d$ on $\mathcal X$.

	\medskip
	
	Different numerical polarizations may induce the same notion of polystability, and hence the same compactified Jacobians. We write  $V(\phi)$ for the subset of $H_d$ of numerical polarizations that define the same notion of polystability as $\phi$. The set of general numerical polarizations is the complement $H_d^\circ$ of the union of all hyperplanes in $H_d$ of the form 
	\[
	H(Y,k) = \bigg \{\phi \, : \, \sum_{X_v \subset Y} \phi_v - \frac{|Y \cap Y^c|}{2} = k  \bigg \}
	\] 
	for integers $k$ and proper subcurves $Y$ such that both $Y$ and $Y^c$ are connected. (To see that it is enough to consider the hyperplanes $H(Y,k)$ such that both $Y$ and $Y^c$ are connected, note that the basic lower bound \eqref{eq:lower bound} is additive on connected components, and one may check $\phi$-semistability and $\phi$-stability equivalently either on $Y$ or $Y^c$, as explained in Remark~\ref{rmk:upper bound}.)
	 Any $V(\phi)$ for $\phi$ general is a connected component of $H_d^\circ$ and the other $V(\phi')$ are the relative interiors of faces of the polytopes $\overline{V(\phi)}$ for general $\phi$ (cf. \cite[Corollary~7.2 and Proposition~7.6]{OdaSeshadri79}). 
	
	\begin{defn}\label{def:dpecialization of polarizations}
		We say that a numerical polarization $\phi$ \emph{specializes to a numerical polarization} $\phi'$ if they have the same degree and $V(\phi') \subset \overline{V(\phi)}$. 
	\end{defn}
	
	\begin{rmk} \label{rmk:v-quasistability}
		One can construct a general numerical polarization $\phi^v$ that specializes to any given polarization $\phi$ via the notion of $v$-quasistability as introduced by Esteves \cite{Esteves01}. Here $v$ is a vertex of $G$ and one may view $v$-quasistability as perturbing $\phi$ infinitesimally in the direction given by $v$ in $H_d \subset \RR^{|V(G)|}$.
		The numerical polarizations obtained in this way are always general \cite[Theorem 2.4 and Remark 2.5]{Esteves09}. Explicitly, $v$-quasistability can be described as follows: a type $(S, \ud)$ is $v$-quasistable, if it is $\phi$-semistable and whenever equality is achieved in the basic lower bound \eqref{eq:lower bound} along a subcurve $Y$ of $X$, then $X_v \subset Y$.
	\end{rmk}
	
	\begin{lma} \label{lma:semistability for varying polarizations}
		Let $\phi$ and $\phi'$ be numerical polarizations of degree $d$ such that $\phi$ specializes to $\phi'$. Then every $\phi$-semistable type is also $\phi'$-semistable.
	\end{lma}
	
	\begin{proof}
		Denote by $V_{\mathrm{ss}}(\phi) \subset H_d$ the set of numerical polarizations $\phi'$ such that every type $(S, \ud)$ that is $\phi$-semistable is $\phi'$-semistable. It suffices to show that $V_{\mathrm{ss}}(\phi)$ is closed, since then $\overline{V(\phi)} \subset V_{\mathrm{ss}}(\phi)$ as claimed. But for any given type $(S, \ud)$ the locus in $H_d$ of polarizations for which it is semistable is closed, since it is a polytope cut out by the finitely many, non-strict inequalities of the basic lower bound \eqref{eq:lower bound}. Then $V_{\mathrm{ss}}(\phi)$ is closed as well, since it is the intersection of finitely many such polytopes, one for each type $(S, \md)$ that is $\phi$-semistable. 
	\end{proof}
	
	\begin{cor} \label{cor:map compactified Jacobians}
		Suppose $\phi$ specializes to $\phi'$. Then the identity map on the general fiber $\Pic^d(\cX_K)$ extends to a surjective morphism between compactified Jacobians \[\overline J_{\cX}(\phi) \to \overline J_{\cX}(\phi').\]
	\end{cor}
	
	\begin{proof}
		By Lemma~\ref{lma:semistability for varying polarizations}, every torsion free rank $1$ sheaf that is $\phi$-semistable is $\phi'$-semistable.  
		Let $W_\phi$ and $W_{\phi'}$ denote the functors of sheaves on $\cX$ that are $\phi$-semistable and $\phi'$-semistable, respectively, as in \cite[\S 4]{Ishida78}. Then we have an inclusion $W_{\phi} \to W_{\phi'}$, and hence a map $\overline J_{\cX}(\phi) \to \overline J_{\cX}(\phi')$ \cite[Theorem 5.6]{Ishida78}. The map is the identity on the generic fiber and has closed image since $\overline J_{\cX}(\phi)$ is proper, thus it is surjective. 
	\end{proof}

	\section{Extensions of line bundles on semistable models} \label{sec:extensions}
	
	Let $\cX_K^\an$ be the nonarchimedean analytification of $\cX_K$, as defined by Berkovich \cite{Berkovich90}. We will recall the skeleton of $\cX_K^\an$ associated to the model $\cX$, and then use this to characterize line bundles on $\cX_K$ that extend to torsion free sheaves of type $(S, \ud)$ on $\cX$.
	For further details on skeletons and semistable models see \cite{BPR-section5}.
	
	\subsection{Skeletons of $\cX_K^\an$ and semistable models} 
	
	Recall that the skeleton of $\cX_K^\an$ associated to the model $\cX \to \Spec(R)$ is a metric graph $\Gamma$, with the dual graph $G$ of the central fiber $X$ as an underlying combinatorial graph. Each edge corresponds to a node of the special fiber, and each node is \'etale locally isomorphic to $xy  = f$ for some $f$ in the maximal ideal of $R$. The length of the edge is then the valuation of $f$, which is the \emph{thickness} of the node.  
		
	A vertex set $V'$ of $\Gamma$ will be a collection of points in $\Gamma$ that are given by the vertices of the dual graph of a semistable model of $\cX_K$ over $\Spec(R)$ that refines our fixed model $\cX$. For a vertex set $V'$ we will write $\cX_{V'}$ for the semistable model having $V'$ as vertex set and $X_{V'}$ for its central fiber.	
	
	The skeleton $\Gamma$ comes with an inclusion $\Gamma \to \cX_K^\an$ and a retraction map 
	\[
	\trop\colon \cX_K^\an \to \Gamma. 
	\]
	This map can be described as follows: a point $x$ of $\cX_K^\an$ is represented by a $K'$-point of $\cX_K$, where $K'/K$ is a valued field extension. Let $\cX_{V'}$ be a semistable model refining $\cX$. Since $\cX_{V'}$ is proper, $x$ extends to an $R'$-point $\Spec(R') \to \cX_{V'}$, where $R'$ is the valuation ring of $K'$. Then for an appropriate choice of semistable model $\cX_{V'}$, the image of the closed point of $\Spec(R')$ lies in the smooth locus of the central fiber, $X_{V'}$. Let $v' \in V'$ be the element of the vertex set such that $X_{v'}$ is the irreducible component of $X_{V'}$ on which this image lies; then $\trop(x) = v'$.
	
	A \emph{divisor} on $\Gamma$ is a finite formal sum $\ud = \sum a_i (p_i)$ with $a_i \in \ZZ$ and $p_i \in \Gamma$, and the degree of a divisor is $\deg(\ud) = \sum a_i$.
	We write $\Div^d(\Gamma)$ for the set of divisors of degree $d$ on $\Gamma$ and $\Div^d(G)$ for the subset supported on vertices of $G$. 
	
	To a piecewise linear function $f$ with integer slopes  on $\Gamma$ we can associate a divisor whose value at a point $p$ is the sum of the outgoing slopes of $f$ at $p$. Divisors obtained in this way are called \emph{principal}. Two divisors $\ud$ and $\ud'$ are said to be linearly equivalent, $\ud \sim \ud'$, if their difference is a principal divisor.  The space $\Pic^d(\Gamma)$ parametrizing classes of divisors of degree $d$, up to linear equivalence, is a torsor over $\Pic^0(\Gamma)$.  Tropicalization gives a natural map
	$
	\trop\colon \Div^d(\cX_K) \to \Div^d(\Gamma),
	$
	which takes algebraic principal divisors to principal divisors on $\Gamma$, and hence induces a map on divisor classes $\Pic^d(\cX_K) \to \Pic^d(\Gamma)$. 
	
	\subsection{Extension of line bundles}
	We now characterize models of $\cX_K$ on which a line bundle $\mathcal{O}(D)$ extends to a line bundle or torsion free sheaf of given type, in terms of $\trop(D)$.  Proposition~\ref{prop:extension-tf} will be essential in the proof of Theorem~\ref{thm:Mumford}.
	
	\begin{lma} \label{lma:extension-lb}
		Let $D$ be a divisor on $\cX_K$, and let $V'$ be a semistable vertex set that contains the support of $\trop(D)$.  Then $\cO(D)$ extends to a line bundle on $\cX' = \cX_{V'}$ whose restriction to $\cX'_v$ has degree $\trop(D)_v$, for each $v \in V'$.
	\end{lma}
	
	\begin{proof}
		If $x$ is a $K$-point in $\cX$ whose tropicalization is contained in $V'$, then the closure of $x$ is a Cartier divisor on $\cX'$.  Hence it suffices to show that if $\trop(D) = 0$ then $\cO(D)$ extends to a line bundle on $\cX'$ with degree zero on each component of the special fiber.
		
		Suppose $\trop(D) = 0$.  Then $\cO(D)$ is a $K$-point of the formal completion $J_0$ of $\Jac(\cX')$ along its special fiber, as in \cite[\S7]{BakerRabinoff15}.  By \cite[Theorem~5.1(c)]{BoschLutkebohmert84} each $K$-point of this formal completion extends to an $R$-point of the relative Jacobian $\Jac(\cX')$, and hence $\cO(D)$ extends to a line bundle on $\cX'$ with degree zero on each component of the special fiber \cite[Theorem~9.3.7]{BLR90}. 
	\end{proof}
	
	\begin{prop} \label{prop:extension-tf}
		The line bundle $\cO(D)$ on $\cX_K$ extends to a torsion free sheaf of type $(S, \ud)$ on $\cX$ if and only if $\trop(D)$ is equivalent to $\ud + \underline e$, where $\underline e$ is the sum of one point in the interior of each edge in $S$.
	\end{prop}
	
	\begin{proof}
		By \cite[Proposition~5.5]{EstevesPacini16}, if a line bundle $L$ on $\cX_K$ extends to a torsion free sheaf of type $(S, \ud)$ on $\cX$ then there is a model $\cX'$ over $\cX$, with precisely one exceptional component over each node in $S$, such that $L$ extends to a line bundle $\cL'$ on $\cX'$ with multidegree $\ud$ on the strict transform of the components of $X$ and degree 1 on each exceptional component.  Such a model corresponds to a semistable vertex set $V'$ consisting of the vertices of $\Gamma$ together with one point in the interior of each edge in $S$. Since $\cX' \to \Spec R$ is projective, we may write $\cL'$ as the difference of two very ample line bundles. Thus, we can write $\cL' = \cO(\mathcal D_1 - \mathcal D_2)$ where the supports of $\mathcal D_1$ and $\mathcal D_2$ do not contain any of the nodes.  Then $D \sim (\mathcal D_1 - \mathcal D_2)_K$ and hence $\trop(D) \sim \trop(\mathcal D_1)_K - \trop(\mathcal D_2)_K = \ud + \underline e$, where $\underline e$ is the sum of the points in $V'$ that are not vertices of $\Gamma$.
		
		For the converse, suppose $\trop(D)$ is equivalent to such a divisor $\ud + \underline e$.  By \cite[Theorem~1.1]{BakerRabinoff15}, $D$ is equivalent to a divisor $D'$ with tropicalization $\ud + \underline e$. Then, by Lemma~\ref{lma:extension-lb}, $L = \mathcal O(D')$ extends to a line bundle $\cL'$ on the model $\cX'$ corresponding to the semistable vertex set $V'$ consisting of the vertices of $\Gamma$ together with the support of $\underline e$, whose restriction to $\cX'_v$ has degree $(\ud + \underline e)_v$.  Hence, the push forward of $\cL'$ to $\cX$ is an extension of $L$ to a torsion free sheaf of type $(S,\ud)$.
	\end{proof}
	
	\begin{defn} \label{def:type_divisor}
		Let $\Gamma$ be a metric graph with underlying graph $G$.  Let $S \subset E(G)$ and $\ud \in \Div(G)$. Then $\D \in \Div(\Gamma)$ is \emph{of type $(S, \ud)$} if $\D = \ud + \underline e$, where $\underline e$ is the sum of one point in the interior of each edge in $S$.
	\end{defn}
	
	We give a convex geometry proof of the following proposition in \S\ref{sec:stabledecomp}, using Namikawa decompositions.  Here we give a short proof via algebraic geometry.
	
	\begin{prop} \label{prop:unique-polystable}
		Let $\D \in \Div(\Gamma)$ and let $\phi$ be a numerical polarization.  Then there is a unique $\phi$-polystable type $(S, \ud)$ such that $\D$ is equivalent to a divisor of type $(S, \ud)$.
	\end{prop}

	\begin{proof}
		After a valued extension of the ground field, we may assume that $\D$ is the tropicalization of a divisor on $\cX_K$.  Existence and uniqueness then follow from Proposition~\ref{prop:extension-tf} and the properness of the compactified Jacobian $\overline J_{\mathcal X}(\phi)$, by the valuative criterion.
	\end{proof}
	
	In the special case of a numerical polarization $\phi^v$ encoding $v$-quasistablilty as in Remark~\ref{rmk:v-quasistability}, Proposition~\ref{prop:unique-polystable} is part of \cite[Theorem 5.6]{AbreuPacini20}, which also proves that the divisor of type $(S,\ud)$ that is equivalent to $\D$ is unique in this case. 
	Note that in general the divisor of type $(S,\ud)$ equivalent to $\D$ is not necessarily unique, as the following example shows. 
	
	\begin{ex} \label{ex:polystable family}
		Let $\Gamma$ be a loop with two vertices $v, w$ of weight one. Assume that both edges have length one. Let $\D$ have degree $0$ on $v$ and degree $2$ on $w$ and define a numerical polarization $\phi$ by $\phi_{v} = \phi_{w} = 1$. Then the type of $\D$, $S = \emptyset$ and $\md = (0,2)$, is $\phi$-semistable but not $\phi$-polystable. The divisors $\D'$ of $\phi$-polystable type that are equivalent to $\D$ have one point in the interior of each edge, both at distance $\epsilon$ from $w$, for $\epsilon \in (0, 1)$, and thus form a one-dimensional family. 
		\begin{figure}[ht]
		    \begin{tikzpicture}[x=0.7pt,y=0.7pt,yscale=-0.8,xscale=0.8]
            \import{./}{figure1.tex}
        \end{tikzpicture}
		    \caption{The one-dimensional family of equivalent divisors in Example~\ref{ex:polystable family}. All divisors are $\phi$-semistable, and the ones in the middle two figures are in addition $\phi$-polystable.}
		    \label{fig:polystable family}
		\end{figure}
		
	\end{ex}
	
	\section{Polytopal decompositions of skeletons}
	\label{section:decompositions}
	
	Let $\Pic^d(\cX_K)^\an$ be the nonarchimedean analytification of $\Pic^d(\cX_K)$ as defined by Berkovich \cite{Berkovich90}. Its skeleton is naturally identified with $\Pic^d(\Gamma)$ \cite{BakerRabinoff15}.  In this section, we describe polytopal decompositions of the skeleton $\Pic^d(\Gamma)$ that are naturally associated to the stability conditions discussed in \S\ref{sec:stability}. We use this to give toric charts on the central fiber of the compactified Jacobian. As an example, we discuss compactified Jacobians $\overline J_{\cX}(\phi^{H,d})$ with $\phi^{H, d}$ a numerical polarization of degree $d = g -1$ or $d = g$ associated to an ample line $H$ as in Remark~\ref{rem:slope stability}, and their relations to the tropical theta divisor and the break divisor decomposition, respectively.  
	
	\begin{rmk}\label{rmk:reduction to degree zero}
		For the remainder of the paper, we fix a basepoint in $\cX(K)$ and use this to identify $\Pic^d(\cX_K)$ with $\Pic^0(\cX_K)$. For simplicity, we assume that the closure of the base point in $\cX$ intersects the central fiber in a smooth point $p \in X_{v'}$. By Remark~\ref{rmk:shift polarization} the identification extends to an identification $\overline J_\cX(\phi)$ with $\overline J_{\cX}(\phi - d p)$. We assume that the tropicalization of the basepoint is a vertex $v' \in V(G)$ and use $v'$ to identify $\Pic^d(\Gamma)$ with $\Pic^0(\Gamma)$. 
	\end{rmk}

	\subsection{Decomposing the tropical Jacobian}
	We briefly recall the theory of tropical Jacobians as developed in \cite{MikhalkinZharkov08, BakerFaber11, BakerRabinoff15} and then define polytopal decompositions given by stability conditions that naturally generalize the Namikawa decompositions for graphs, from \cite[\S6]{OdaSeshadri79}, to metric graphs with arbitrary edge lengths. 
	As before, we set $\Gamma = \trop(\cX_K^{\mathrm{an}})$ and write $G$ for the dual graph of $X$.
	
	\subsubsection{The integration pairing}
	Let $C_0(G, A)$ and $C_1(G, A)$ be the 0- and 1-chains on $G$, with coefficients in $A$, and let $\partial\colon C_1(G, A) \rightarrow C_0(G, A)$ be the boundary map. 
	\begin{rmk}\label{rmk:numerical polarization OS}
		The image of the boundary map $\partial$ consists of formal sums of vertices of $G$ with coefficients in $A$ that sum to $0$.
		In particular, we may view a numerical polarization $\phi$ of degree $d = 0$ as an element in $\partial(C_1(G, \QQ))$.
	\end{rmk}

	Let $H_1(G, \RR)$ be the first homology group of $G$ with real coefficients. It is a vector space of dimension $b_1(G)$, the first Betti number of $G$. Then $$H_1(G, \ZZ) = H_1(G, \RR) \cap C_1(G, \ZZ).$$
	Recall that we denote by $\ell\colon E(G) \to \RR_+$ the length function on $G$ giving $\Gamma$.
	
	Recall that a $1$-form $\psi$ on $G$ is a formal sum with integer coefficients $\sum \psi_{\vec e} d \vec e$ with $\vec e$ an oriented edge of $G$, where the value on the edge with the opposite orientation to $\vec e$ is required to be $- \psi_{\vec e}$. A $1$-form $\psi$ on $G$ is called harmonic, if for every vertex $v \in V(G)$ we have $\sum \psi_{\vec e} = 0$ where the sum is over all $\vec e \in \Star(v)$ oriented away from $v$. We have an integration pairing $\int$ between $1$-forms on $G$ and $C_1(G, \RR)$. For simplicity, we fix an orientation on each edge of $G$, and then
	\[
	\int_{e'} de=
	\begin{cases}
	\ell(e) \text{ if } e = e',\\
	0 \text{ otherwise},\\
	\end{cases}
	\]
	for $e, e' \in E(G)$.  Let $\Omega(\Gamma)$ denote the real vector space of harmonic $1$-forms on $\Gamma$ and $\Omega(\Gamma)^*$ its dual. The restriction of the integration pairing to $\Omega(\Gamma) \times H_1(G, \RR)$ is perfect, and hence induces a natural identification $$\mu \colon H_1(G, \RR) \xrightarrow{\sim} \Omega(\Gamma)^*.$$ 
	
	\subsubsection{Two lattices} The isomorphism $\mu$ gives rise to a natural lattice
	\[
	\Lambda = \mu(H_1(G, \ZZ)),
	\]
	of full rank in $\Omega(\Gamma)^*$. We will also consider another such lattice, as follows.
	
	The vector space $C_1(G,\RR)$ is self-dual with respect to the edge length pairing, given on basis elements by
	\[
	\left\langle e,e' \right\rangle_\ell=
	\begin{cases}
	\ell(e) \text{ if } e = e',\\
	0 \text{ otherwise}.\\
	\end{cases}
	\]
	We then define
	\[
	\Lambda^* = \left \{ \mu(v) \in \Omega(\Gamma)^* : \langle u,v \rangle_\ell \in \ZZ \mbox{ for all } u \in H_1(G, \ZZ) \right \}.
	\]
	Note that the isomorphism
	$
	\Lambda \cong H_1(G,\ZZ)$ induces an isomorphism of dual lattices $\Lambda^* \cong H^1(G, \ZZ).
	$ 
	Note that, while $\Lambda$ and $\Lambda^*$ are lattices of full rank in $\Omega(\Gamma)^*$, they need not be commensurable when the edge lengths are not rational.
	
	\begin{defn}
		Let 
		$
		\Sigma$ be the tropical Jacobian, that is, the real torus
		\[
		\Sigma = \Omega(\Gamma)^* / \Lambda,
		\]
		considered with the integer affine structure induced by $\Lambda^*$. 
	\end{defn}
	
	\subsubsection{Base points and Abel-Jacobi}
	
	Recall that we denote by $\Pic^d(\Gamma)$ the tropical Picard variety of degree $d$ divisors on $\Gamma$ modulo linear equivalence. Tropicalization takes algebraic principal divisors to tropical principal divisors and descends to
	\[
	\trop \colon \Pic^d(\cX)(K) \to \Pic^d(\Gamma),
	\]
	which agrees with retraction to the skeleton.  Our fixed choice of a basepoint in $\cX(K)$ induces identifications of $\Pic^d$ with $\Pic^0$, for both $\cX_K$ and $\Gamma$. 
	
	In close analogy with the Abel-Jacobi theory in algebraic geometry, one can identify the tropical Jacobian $\Sigma$ with the tropical Picard variety $\Pic^0(\Gamma)$ via the tropical Abel-Jacobi map. We briefly recall the basic outline of this construction and refer to \cite{MikhalkinZharkov08, BakerFaber11, BakerRabinoff15} for further details. If $x_1, \ldots, x_d$ are points in $\Gamma$, then we can choose paths $\xi_1, \ldots, \xi_d$ in $\Gamma$, where $\xi_i$ is a path from the basepoint to $x_i$.  Integrating harmonic forms along these paths gives an element of $\Omega(\Gamma)^*$.  Of course, these integrals depend on the choice of paths, but the differences are in $\Lambda = \mu( H_1(G,\ZZ))$.  In particular, the resulting map $\Gamma^d \to \Sigma$ is well-defined and induces an isomorphism $\Pic^0(\Gamma) \simeq \Sigma$. See \cite[\S 6]{MikhalkinZharkov08} and \cite{BakerFaber11}.  Finally, $\Pic^0(\Gamma)$ is naturally identified with the skeleton of $\Pic^0(\cX_K)^\an$, and the natural diagrams relating the algebraic and tropical Abel-Jacobi maps commute \cite{BakerRabinoff15}. 
	
	\subsubsection{Admissible polytopal decompositions}
	We next recall the definition of periodic polytopal decompositions of $\RR^n$ that are admissible with respect to the value group $|K^\times|$.  Note that this definition depends on two lattices: the first is the lattice that determines which polytopes are admissible.  For simplicity we fix this to be $\ZZ^n$.
	\begin{defn} \label{def:admissible}
		A polytope in $\RR^n$ is \emph{admissible} if it is an intersection of closed halfspaces $\{ u \in \RR^n : \langle u, v \rangle \geq a \}$ with $v \in \ZZ^n$ and $a$ in the value group $|K^\times|$.  
	\end{defn}
	
	If the affine span of each face of a polytope has rational slope, then it is defined by inequalities $\langle \cdot, v \rangle \geq a$ with $v \in \ZZ^n$. If the vertices of the polytope are in addition contained in $|K^\times|^n$, then also $a \in |K^\times|$. Indeed, each of the supporting hyperplanes $\{ u \in \RR^n : \langle u, v \rangle = a \}$ contains a vertex, and so $a$ is a sum of elements in $|K^\times|$ and thus itself in $|K^\times|$. In particular, any such polytope is admissible.
	
	\medskip
	
	The second lattice is the one with respect to which the decomposition is periodic, which we denote by $\Lambda \subset \RR^n$. We emphasize that $\Lambda$ need not be commensurable with the sublattice $\ZZ^n \subset \RR^n$ that controls admissibility.
	
	\begin{defn} \label{def:poldec}
		An \emph{admissible polytopal decomposition} of $\RR^n$ is a locally finite collection $\Delta = \{Q_i\}_i$ of admissible polytopes such that:
		\begin{enumerate}
			\item \label{def:poldec1} $\bigcup_i Q_i =\RR^n$
			\item \label{def:poldec2} Any face of $Q_i$ belongs to $\Delta$. 
			\item \label{def:poldec3} For $Q_i, Q_j \in \Delta$, $Q_i \cap Q_j$ is either empty or a common face of $Q_i$ and $Q_j$. 
		\end{enumerate}
		An admissible polytopal decomposition is called \emph{$\Lambda$-periodic}, if we have in addition:
		\begin{enumerate}[resume]
			\item \label{def:poldec4} For every $\lambda \in \Lambda$ and $Q_i \in \Delta$, $Q_i + \lambda \in \Delta$.
			\item \label{def:poldec5} there are finitely many classes of $Q_i$ modulo translations by elements of $\Lambda$.
		\end{enumerate}
	\end{defn}
	An admissible $\Lambda$-periodic polytopal decomposition $\Delta$ of $\RR^n$ induces a decomposition of the real torus $\RR^n/\Lambda$ as a union of closed subsets $$\RR^n/\Lambda = \bigcup \pi(Q_i),$$ where $Q_i$ ranges over a set of representatives of the finitely many $\Lambda$-orbits in $\Delta$, and $\pi\colon \RR^n \to \RR^n/\Lambda$ is the natural projection.  We refer to any decomposition of $\RR^n/\Lambda$ that arises in this way as an \emph{admissible polytopal decomposition}.

	\subsection{Decompositions into polystable types} \label{sec:stabledecomp}
	Recall that, given a divisor $\ud$ on the underlying graph $G$ and $S \subset E(G)$, the divisors of type $(S,\ud)$ on $\Gamma$ are those that can be expressed as $\ud + \underline e$, where $\underline e$ is the sum of one point in the interior of each edge in $S$.  Let $d = \deg (\ud) + |S|$.
	
	\begin{defn}
		Let $\theta^\circ_{S,\ud} \subset \Pic^d(\Gamma)$ be the set of classes of divisors of type $(S,\ud)$, and let $\theta_{S, \ud}$ be its closure.
	\end{defn}
	
	\begin{prop} \label{prop:polytopal}
		Choose an identification $\Omega(\Gamma)^* \cong \RR^n$ that induces $\Lambda^* \cong \ZZ^n$, and use the fixed basepoint to identify $\Sigma \cong \Pic^d(\Gamma)$.  Then
		$$
		\Delta_{\phi} = \left\{ \theta_{S,\ud} : (S, \ud) \mbox{ is } \phi \mbox{-polystable} \right\}
		$$
		is an admissible polytopal decomposition of $\Sigma$.
	\end{prop}
	
	We will refer to $\Delta_{\phi}$ as the \emph{Namikawa decomposition} into $\phi$-polystable types. As the notation indicates, it depends on the numerical polarization $\phi$. In the case where all edge lengths are 1, it is an instance of the Namikawa decompositions for graphs studied in \cite[\S6]{OdaSeshadri79}.
	
	\begin{proof}[Proof of Proposition~\ref{prop:polytopal}]
		The conclusion does not depend on the choice of coordinates $\Omega(\Gamma)^* \cong \RR^n$, as long as it identifies $\Lambda^*$ with $\ZZ^n$.  We make such a choice by fixing a spanning tree for $G$.  Let $f_1, \ldots, f_n$ be the edges in the complement of this spanning tree.  Then
		$\big \{ \frac{1}{\ell(f_i)} \int_{f_i} : 1 \leq i \leq n \big \}$ is a basis for $\Lambda^*$, and we use these as our coordinates.
		
		We begin by verifying that each $\theta_{S, \ud}$ is the image of an admissible polytope in $\Omega(\Gamma)^*$.  Let $\xi \in C_1(G,\ZZ)$ be the $1$-chain associated to some choice of paths from the basepoint to the points of $\ud$.  Then $\theta_{S, \ud}$ is the image in $\Sigma$ of the polytope 
		\begin{equation} \label{eq:QSd}
			Q_{S, \ud}(\xi) = \bigg\{\int_\xi + \sum_{e_i \in S} \frac{t_i}{\ell(e_i)} \int_{e_i} \ : \ 0 \leq t_i \leq \ell(e_i)\bigg\},
		\end{equation}
		in $\Omega(\Gamma)^*$.  The edges of $Q_{S, \ud}(\xi)$ are spanned by the vectors $\frac{1}{\ell(e_i)}\int_{e_i}$, for $e_i \in S$, which are in $\Lambda^*$.  It follows that the affine span of each face has rational slope. Furthermore, each vertex of $Q_{S, \ud}(\xi)$ is an integer linear combination of the vectors $\int_{e_i}$, which lie in $\Lambda^* \otimes_\ZZ |K^\times|$. Hence, in our chosen identification $\Omega(\Gamma)^* \cong \RR^n$ the vertices of $Q_{S,\ud}(\xi)$ are in $|K^\times|^n$. By the discussion after Definition~\ref{def:admissible}, it follows that $Q_{S, \ud}(\xi)$ is admissible, as required.
		
		It remains to show that $\{ Q_{S, \ud}(\xi) : (S, \ud) \mbox{ is } \phi \mbox{-polystable} \}$ forms a $\Lambda$-periodic polytopal decomposition of $\Omega(\Gamma)^* \cong \RR^n$, i.e., that it satisfies properties (1)-(5) from Definition~\ref{def:poldec}.  In fact, this decomposition is an instance of the Namikawa decompositions of a subspace of a vector space with a lattice and a quadratic form, from \cite[\S1]{OdaSeshadri79}.  We briefly recall the original construction, mirroring their notation.  Consider the real vector space $E = C_1(G,\RR)$ with inner product given by the edge length pairing $\langle \cdot, \cdot \rangle_\ell$, and extend the lattice $\Lambda = \mu(H_1(G,\ZZ))$ to $\Lambda_E = C_1(G,\ZZ)$.  Note that the Delaunay and Voronoi decompositions of $E$ associated to $\Lambda_E$ and $\langle \cdot, \cdot \rangle_\ell$ are independent of the edge lengths.  By construction, each polytope $Q_{S, \ud}(\xi)$ is the orthogonal projection to $E' = \Omega(\Gamma)^* \cong H_1(G,\RR)$ of the Delaunay polytope $D_{S,\ud}(\xi) = \xi + \sum_{e \in S} [0,e]$.
		
		Now, consider the special case where $\ell = \mathbf{1}$, i.e., all edge lengths are equal to $1$. Using the identification of Remark~\ref{rmk:reduction to degree zero}, we may assume $d = 0$. In this case the numerical polarization $\phi$ is in $\partial(C_1(G,\RR))$ (cf. Remark~\ref{rmk:numerical polarization OS}), and we let $\psi$ be the unique element in the orthogonal complement of $E'$ with respect to $\langle \cdot,\cdot \rangle_{\mathbf 1}$ such that $\partial \psi = \phi$.  By \cite[Proposition~6.5]{OdaSeshadri79}, $E' + \psi$ meets the relative interior of the Voronoi cell dual to $D_{S, \ud} (\xi)$ if and only if $(S, \ud)$ is $\phi$-polystable.
		
		Finally, with $\psi \in C_1(G, \RR)$ fixed, we return to the case of arbitrary edge lengths.  Let $\psi'$ be the unique point in the intersection of $E' + \psi$ with the orthogonal complement of $E'$ with respect to $\langle \cdot,\cdot \rangle_\ell$.  By construction, $E' + \psi' = E'+ \psi$.  In particular, $E' + \psi'$ meets the relative interior of the Voronoi cell dual to $D_{S, \ud} (\xi)$ if and only if $(S, \ud)$ is $\phi$-polystable.  Therefore, applying \cite[Proposition~1.6]{OdaSeshadri79}\footnote{In \cite[\S1]{OdaSeshadri79}, it is assumed that the projection of the lattice in $E$ to the subspace $E'$ is again a lattice, which is not satisfied in our setting, for arbitrary $\ell$.  However, as explained in \cite[p.~125]{Oda13}, no such assumption is needed for the proof of the cited proposition.} to $E = C_1(G,\RR)$, with respect to the inner product $\langle \cdot, \cdot \rangle_\ell$ and the parameter $\psi'$, we see that the orthogonal projections $Q_{S,\ud}(\xi)$ of these Delaunay polytopes form a $\Lambda$-periodic polytopal decomposition of $E' = \Omega(\Gamma)^*$, as required.
	\end{proof}

	\begin{rmk}
		See \cite[\S8]{OdaSeshadri79}, \cite[Example 1.4]{ABKS}, and \cite[\S5.2]{MolchoWise18} for explicit examples of Namikawa decompositions $\Delta_{\phi}$.  Note also that the existence and uniqueness result stated in Proposition~\ref{prop:unique-polystable}, and proved there using algebraic geometry, is an immediate consequence of Proposition~\ref{prop:polytopal}.
	\end{rmk}
	
	Recall that we discussed specialization of numerical polarizations in Section~\ref{sec:varying polarization}. 
	
	\begin{cor}\label{cor:refinement}
		Suppose the numerical polarization $\phi$ specializes to $\phi'$. Then the Namikawa decomposition into $\phi$-polystable types $\Delta_{\phi}$ is a refinement of the Namikawa decomposition into $\phi'$-polystable types $\Delta_{\phi'}$.
	\end{cor}
	
	\begin{proof}
		As explained in the proof of Proposition~\ref{prop:polytopal}, $\Delta_{\phi}$ and $\Delta_{\phi'}$ are instances of Namikawa decompositions and the claim then is \cite[Proposition 2.3 (iv)]{OdaSeshadri79}. 
	\end{proof}
	
	\subsubsection{Toric charts} \label{sec:toriccharts}
	
	Choose coordinates $\Omega(\Gamma)^* \cong \RR^n$ so that $\Lambda^*$ is identified with $\ZZ^n$.  Let $\Delta$ be an admissible $\Lambda$-periodic polytopal decomposition of $\Omega(\Gamma)^*$.  Admissibility ensures that, for each vertex $v$ in $\Delta$ the local fan $\Star(v)$ is rational.  Then we can associate to $\Delta$ a stable toric variety $Y_\Delta$, whose irreducible components correspond to vertices of $\Delta$, as follows.
	
	Since the fan $\Star(v)$ is rational (with respect to $\ZZ^n = \Lambda^*$), it gives rise to a normal toric variety $Y_v$, with dense torus $T = \Lambda^* \otimes_\ZZ k^*$.  
	Faces of dimension $k$ in $\Delta$ that contain $v$ correspond to torus fixed subsets of codimension $k$ in $Y_v$. We obtain $Y_\Delta$ by gluing the $Y_v$ along such torus fixed subsets according to the containment relations of $\Delta$. This stable toric variety $Y_\Delta$ is locally of finite type, and comes equipped with a $\Lambda$-action, induced by the action of $\Lambda$ on $\Omega(\Gamma)^*$ via translation.  This lattice action commutes with the torus action, and $Y_\Delta$ has only finitely many $(\Lambda \times T)$-orbits.  In particular, we have such a $\Lambda$-periodic toric variety associated to
	\[
	\widetilde \Delta_{\phi} = \left \{ Q_{S, \ud}(\xi) : (S, \ud) \mbox{ is } \phi \mbox{-polystable} \right \},
	\]
	where $Q_{S,\ud}(\xi)$ is the polytope defined in \eqref{eq:QSd}.

	Next, recall that the generalized Jacobian of $X$ parametrizes line bundles with degree zero on every component, and fits into a short exact sequence
	\begin{equation} \label{eq:Jacnu}
		1 \to H^1(G, k^*) \to \Jac(X) \to \Jac(X^\nu) \to 1.
	\end{equation}
	As discussed above, the edge length pairing induces an identification  $\Lambda^* \cong H^1(G, \ZZ)$. Hence we may identify the torus $T$ with $H^1(G, k^*)$ and set 
	\[
	E_v = Y_v \times_{T} \Jac(X) \mbox{ \ and \ } E_\Delta = Y_\Delta \times_{T} \Jac(X).
	\]
	Note that $E_\Delta$ carries a natural action of $\Lambda$.  The quotient is semi-toroidal, of finite type, and has only finitely many $\Jac(X)$-orbits.

	\begin{prop} \label{prop:specialquotient}
		There is a natural isomorphism identifying the compactified Jacobian of the special fiber $X$ with the quotient
		\[
		E_{\widetilde \Delta_{\phi}}/\Lambda \cong \overline J_X(\phi).
		\]
	\end{prop}
	
	\begin{proof}
		The normal fans of the polytopes $Q_{S, \ud}(\xi)$, with respect to the lattice $\Lambda^* = H^1(G,\ZZ)$, are independent of the edge lengths, as are the incidence relations among them.  It follows that $E_{\widetilde \Delta_{\phi}}$ is independent of the edge lengths.  We may therefore assume all edge lengths are 1. As before, by the identification of Remark~\ref{rmk:reduction to degree zero} we may assume $d = 0$. In this case, the quotient $E_{\widetilde \Delta_{\phi}}/\Lambda$ is the compactified Jacobian for the numerical polarization $\phi$, by \cite[Theorem 13.2]{OdaSeshadri79}. 
	\end{proof}

	\subsection{Example: Degree $d = g-1$ and the tropical Theta divisor} \label{sec:break}
	
	Let $\phi^{H, g-1}$ be the numerical polarization of degree $g -1$ that is induced by an ample line bundle $H$ as in Remark~\ref{rem:slope stability}.
	As observed by Alexeev \cite{Alexeev04}, $\phi^{H, g-1}$ does not depend on $H$ in this case; inserting $d + 1 - g = 0$ in \eqref{eq:slope inequality} gives  \[
	\left(\phi^{H, g-1}\right)_v = g(X_v) - 1 + \frac{|X_v \cap X_v^c|}{2}.\]
	
	For any subcurve $Y \subset X$ and set of nodes $S \subset E$ we have
	\begin{equation} \label{eq:genus_partial_normalization}
		g(Y_S^\nu) - 1  = \left(\sum_{X_v \subset Y} g(X_v) - 1 + \frac{|X_v \cap X_v^c \cap Y^{\mathrm{sing}}|}{2}\right) - |Y^{\mathrm{sing}} \cap S|, 
	\end{equation}
	where $Y_S^\nu$ is the normalization of $Y$ at nodes in $S \cap Y^{\mathrm{sing}}$. Thus we get by the basic lower bound \eqref{eq:lower bound} that a type $(S, \ud)$ is $\phi^{H, g-1}$-semistable if and only if 
	\begin{equation} \label{eq:lower bound g - 1}
		\sum_{X_v \subset Y} \ud_v \geq g(Y_S^{\nu}) - 1
	\end{equation}
	for every subcurve $Y \subset X$. The type is $\phi^{H, g-1}$-stable if and only if the above inequality is strict for every proper subcurve. The multidegrees that are $\phi^{H, g-1}$-semistable and $\phi^{H, g-1}$-stable are closely related to orientations on the dual graph, as we now recall.
	
\begin{defn} \label{def:orientable divisor}
    Let $O$ be an orientation of the edges of $G$. Then $\md_O$ is the multidegree on $G$ with value 
    \begin{equation} \label{eq: orientable divisor}
        \mathrm{indeg}_v(O) - 1 + g(X_v)
    \end{equation}
	at a vertex $v \in V(G)$. Here $\mathrm{indeg}_v(O)$ denotes the number of edges that are oriented towards $v$ in $O$. 
\end{defn} 

\noindent A multidegree is called \emph{orientable}, if it is of the form $\md_O$ for some orientation $O$. More generally, a  divisor $\D$ on the metric graph $\Gamma$ is called orientable if it is given by an orientation on the underlying combinatorial graph with vertex set $V' = V(G) \cup \supp(\D)$. Any orientable multidegree has degree $g - 1$.  The following is \cite[Proposition 3.6]{Alexeev04}:
	
	\begin{lma}\label{lma: Hakimi}
		A multidegree $\md$ is $\phi^{H, g-1}$-semistable if and only if it is orientable. It is strictly $\phi^{H, g-1}$-semistable along a subcurve $Y$ if and only if all edges corresponding to nodes in $Y \cap Y^c$ are oriented away from the subgraph of $G$ corresponding to $Y$.
	\end{lma}
	
	\noindent A set of edges corresponding to $Y \cap Y^c$ as in the second statement of Lemma~\ref{lma: Hakimi} is called a \emph{directed cut} of $O$.

	\medskip
	
	To describe the Namikawa decomposition of $\Pic^{g - 1}(\Gamma)$ into $\phi^{H, g-1}$-polystable types, we need one more definition: Recall that a divisor class on $\Gamma$ is called effective if it contains a divisor with non-negative coefficients. We write $\underline g$ for the divisor on $G$ with value $g(X_v^\nu)$ on a vertex $v$ and define the \emph{tropical Theta divisor} as
	\[ \Theta^{\trop} = \Big \{[\D] \in \Pic^{g - 1}(\Gamma)\, : \, \left [\D - \underline g \right ] \mbox{ is effective} \Big \}. \]
	In particular, if $X$ is a union of rational curves, i.e., if $\underline g = \underline 0$, then $\Theta^{\trop}$ coincides with the effective locus in $\Pic^{g-1}(\Gamma)$.
	
	\begin{prop} \label{prop:tropical theta}
		There is a unique maximal dimensional cell $\Delta_{(S_0, \ud_0)}$ in the Namikawa decomposition $\Delta_{\phi^{H, g-1}}$ and \[\Delta_{(S_0, \ud_0)} = \Pic^{g-1}(\Gamma) \smallsetminus \Theta^{\trop}.\]
	\end{prop}
	
	\begin{proof}
		We first reduce to the case that $\underline g$ is trivial. Consider the translation via $\underline g$ \[t_{\underline g}\colon \Pic^{g -1}(\Gamma) \to \Pic^{b_1(G) - 1}(\Gamma), \, \D \mapsto \D - \underline g, \] where $b_1(G)$ denotes the first Betti number of $G$. Let $G'$ and $\Gamma'$ be obtained from $G$ and $\Gamma$ by setting all weights on vertices to zero. Their genus is $b_1(G)$. 
		
		Suppose $O$ is an orientation on $\Gamma$ with associated divisor $\D = \md_O$. Then we may view $O$ also as an orientation on $\Gamma'$ whose associated divisor by definition is $t_{\underline g}(\D)$, see \eqref{eq: orientable divisor}. We identify $\Pic^{g-1}(\Gamma)$ and $\Pic^{b_1(G)-1}(\Gamma')$ via $t_{\underline g}$ and the natural identification $\Pic^{b_1(G)-1}(\Gamma) \simeq \Pic^{b_1(G)-1}(\Gamma')$. In particular, Lemmas~\ref{lma:polystable_char} and \ref{lma: Hakimi} together with the above discussion imply that the Namikawa decomposition of $\Pic^{g -1}(\Gamma)$ with respect to $\phi^{H,g-1}$ is identified with the Namikawa decomposition of $\Pic^{b_1(G) -1}(\Gamma')$ with respect to $\phi^{H,b_1(G) -1}$, viewed as a numerical polarization on $G'$. Thus it suffices to show the claim for $G'$ and $\Gamma'$.

		So assume that $g(X_v^\nu) = 0$ for all $v \in V(G)$. By \cite[Proposition 7.8]{OdaSeshadri79}, there is a unique $g$-dimensional cell $\Delta_{(S_0, \ud_0)}$ in $\Delta_{\phi^{H, g-1}}$. Concretely, $S_0 = E(G)$ and $\ud_0$ has value $-1$ on each vertex. We need to show that a divisor class of degree $g-1$ is effective if and only if its $\phi^{H, g-1}$-polystable representatives are not of type $(S_0, \ud_0)$.
		
		Suppose first, that $\D$ is a divisor of type $(S, \ud)$ on $\Gamma$ with $(S, \ud)$ a $\phi^{H, g-1}$-polystable type that is different from $(S_0, \ud_0)$. By Lemmas~\ref{lma:polystable_char} and \ref{lma: Hakimi} we have $\md = \md_O$, with $O$ an orientation on $G - S$ that contains no directed cuts. Since $S_0 \neq S$, we have $E(G - S) \neq \emptyset$ and we may choose $v$ adjacent to an edge of $G - S$. Then not all adjacent edges are oriented away from $v$ in $O$, since they would otherwise form a directed cut. Thus $\md_v = \mathrm{indeg}_v(O) - 1 \geq 0$, and thus also $\D_v \geq 0$. 
		
		We extend $O$ to an orientation $\widehat O$ on $\Gamma$ giving $\D$ as follows: each edge not contained in $S$ is oriented as in $O$; for each edge $e$ in $S$, $\D$ contains a unique point $p$ in its interior and the two intervals of $e \smallsetminus p$ are oriented both towards $p$. By \cite[Theorem 4.18]{ABKS} \[\D = \md_{\widehat O} \sim \md_{O'}\] with $O'$ $v$-connected, that is, each point can be reached via a directed path from $v$. In particular, $\md_{O'}$ has value $\mathrm{indeg}_{v'}(O') - 1 \geq 0$ for all $v' \neq v$. On the other hand, the construction of $O'$ in the proof of \cite[Theorem 4.18]{ABKS} affects only directed cuts. By what we said above, at least one edge $e$ of $G - S$ is oriented towards $v$ and $e$ is not contained in a directed cut. It follows that $e$ is still oriented towards $v$ in $O'$. Hence $\md_{O'}$ is non-negative also at $v$ and thus effective.
		
		It remains to show that if $\D$ is of type $(S_0, \ud_0)$, then $\D$ is not effective. In this case, every edge of a vertex set that contains the support of $\D$ is contained in a directed cut of $\widehat O$, for $\D = \md_{\widehat O}$ constructed as above. It is well known, that $\D$ then is not equivalent to an effective divisor; for example, the argument in the proof of \cite[Theorem 4.4]{Backman17} works for metric graphs as well.
	\end{proof}
	
	\subsection{Example: Degree $d = g$ and the break divisor decomposition} The following is a straightforward adaptation of the definition of break divisors in \cite{MikhalkinZharkov08} to the case of vertex weighted graphs:
	\begin{defn}
		Let $G$ be a graph. A divisor $\ud$ on $G$ is called a break divisor if there is a spanning tree $G'$ of $G$ and a map $\varphi\colon E(G) \smallsetminus E(G') \to V(G)$ such that $\varphi(e)$ is adjacent to $e$ and 
		\[
		\ud \, = \, \underline g  \ + \sum_{e \in E(G) \smallsetminus E(G')} \varphi(e).
		\]
	\end{defn}

	Recall from Remark~\ref{rem:slope stability} that an ample line bundle $H$ on $X$ induces a numerical polarization $\phi^{H,d}$ in degree $d$. Recall furthermore, that $\phi^{H,d}$ is called general if the basic lower bound is not an integer for any subcurve. In this case, the notions of $\phi^{H,d}$-semistability, $\phi^{H,d}$-stability and $\phi^{H,d}$-polystability coincide. Finally, recall that we denote by $V\left(\phi^{H,d}\right)$ the set of numerical polarizations that induce the same notion of polystability as $\phi^{H,d}$. See Section~\ref{sec:varying polarization}.
	
	\begin{lma}\label{lma:polystable_break}
		Let $\phi^{H,g}$ be a numerical polarization of degree $g$ induced by an ample line bundle $H$ on $X$. Then 
		\begin{enumerate}
		    \item \label{lma:polystable_break1} $\phi^{H,g}$ is general,
		    \item \label{lma:polystable_break2} a type $(S, \ud)$ with $|S| + \deg(\md) = g$ is $\phi^{H,g}$-polystable if and only if $\ud$ is a break divisor on $G - S$, and
		    \item \label{lma:polystable_break3} $V(\phi^{H,g})$ and the corresponding compactified Jacobian $\overline J_{\cX}\left(\phi^{H,g}\right)$ do not depend on $H$.
		\end{enumerate}
	\end{lma}
	\begin{proof}
		Inserting the definition \eqref{eq:slope inequality} of the numerical polarization $\phi^{H,g}$ induced by $H$ in the basic lower bound \eqref{eq:lower bound} for a type $(S, \ud)$  and  a proper subcurve $Y \subset X$ becomes in degree $g$
		\begin{equation} \label{eq:upper_g}
			\sum_{X_v \subset Y} \ud_v \geq g(Y^\nu_{S}) - 1 + \frac{\deg (H|_Y)}{\deg (H)},
		\end{equation}
		where we used again \eqref{eq:genus_partial_normalization}. Since $H$ is ample, we have $0 < \frac{\deg (H|_{Y})}{\deg (H)} < 1$ and the right hand side of \eqref{eq:upper_g} is not an integer. Thus $\phi^{H,g}$ is a general numerical polarization, as claimed in \eqref{lma:polystable_break1}.
		
		It follows furthermore, that the type $(S, \ud)$ is $\phi^{H,g}$-polystable if and only if 
		\[
		\sum_{X_v \subset Y} \ud_v \geq g(Y^\nu_{S})
		\]
		for every proper subcurve $Y \subset X$.
		By \cite[Lemma 3.3 and Proposition 4.11]{ABKS} this is equivalent to $\ud$ being a break divisor on $G - S$. This shows \eqref{lma:polystable_break2}, from which~\eqref{lma:polystable_break3} immediately follows.
	\end{proof}
	\noindent In particular, the Namikawa decomposition of $\Pic^g(\Gamma)$ into $\phi^{H,g}$-polystable types coincides with the break divisor decomposition introduced in \cite{ABKS} for any $H$. 
	
	\subsubsection{Twisting by a section}
	
	Finally, we relate the two examples for $d= g$ and $d = g- 1$. Namely, combinatorial results of \cite{ABKS} give a close relation between the two compactified Jacobians via the notion of $v$-quasistability in the sense of Esteves. 
	
	As in Remark~\ref{rem:slope stability}, we denote by $\phi^{H,d}$ the numerical polarization encoding slope stability with respect to $H$. In addition, let $\phi^{H,d, v}$ denote the numerical polarization encoding $v$-quasistability with respect to $\phi^{H,d}$ as in Remark~\ref{rmk:v-quasistability}.
	
	\begin{prop} \label{prop:degreeg}
		Let $H$ be an ample line bundle on $X$ and
		$\sigma\colon \Spec(R) \to \cX$ be a section whose image intersects the central fiber in a smooth point on the irreducible component $X_v$. Then twisting by $\sigma$ induces an isomorphism
		\[
		\overline J_\cX\left(\phi^{H,g}\right) \simeq \overline J_\cX\left(\phi^{H, g-1, v}\right).
		\]
	\end{prop}
	\begin{proof}
		The section $\sigma$ gives a Cartier divisor of relative degree 1 on the total space $\cX$. 
		We claim that $\cF(-\sigma)$ is $v$-quasistable with respect to $\phi^{H, g-1}$, for $[\cF] \in \overline J_\cX(\phi^{H, g})$. Indeed, by Lemma~\ref{lma:polystable_break} and \cite[Lemma 3.3]{ABKS}, the multidegree $\md$ of $\cF(-\sigma)$ restricted to the central fiber is $v$-orientable, defined as in the proof of Proposition~\ref{prop:tropical theta}. In particular, any directed cut in the orientation giving $\md$ needs to be oriented away from $v$, and thus $\md$ is $v$-quasistable with respect to $\phi^{H,g-1}$ by Lemma~\ref{lma: Hakimi}.
		Hence, by the modular properties of $\overline J_\cX\left(\phi^{H,g}\right)$ and $\overline J_\cX\left(\phi^{H,g-1, v}\right)$, the map $[\cF] \mapsto [\cF(-\sigma)]$ gives a well-defined morphism $\overline J_\cX\left(\phi^{H,g}\right) \to \overline J_\cX\left(\phi^{H,g-1,v}\right)$.  A similar argument shows that $[\cF'] \mapsto [\cF'(\sigma)]$ gives an inverse morphism.
	\end{proof}
	
	\section{Mumford models} \label{sec:Mumford}
	
	In this final section, we prove the main results stated in the introduction by showing that each compactified Jacobian of $\cX$ is the Mumford model of $\Pic^d(\cX_K)^\an$ associated to the Namikawa decomposition of its skeleton $\Pic^d(\Gamma)$ into $\phi$-polystable types.
	We begin by briefly recalling Mumford's construction of formal models of abelian varieties via uniformization \cite{Mumford72b}, following the presentation of \cite[\S4]{Gubler10}, to which we refer the reader for details and further references.

	\subsection{Uniformization} \label{sec:uniformization}	Let $\cA$ be an abelian variety over $K$.  Then there is an analytic domain $A_1 \subset \cA^\an$ with a formal model $\mathfrak{A}_1$ whose special fiber is semiabelian (i.e., an extension of an abelian variety by a multiplicative torus), a formal abelian scheme $\mathfrak{B}$ with generic fiber $B$, and an exact sequence of analytic groups
	\begin{equation}\label{eq:Raynaud1}
		1 \to T_1 \to A_1 \xrightarrow{q_1} B \to 1,
	\end{equation}
	where $T_1$ is the maximal formal affinoid subtorus of $A_1$.  
	
	We identify $T_1$ with the compact subgroup of $T = (\mathbb{G}_m^n)^\an$ where all coordinates have valuation zero, and let $E = (A_1 \times T) / T_1$.  Then there is a canonical surjective \emph{uniformization map} of analytic groups $E \to \mathcal A^\an$ with discrete kernel $M$, and an exact sequence, called the Raynaud extension
	\begin{equation}\label{eq:Raynaud}
		1 \to T \to E \xrightarrow{q} B \to 1
	\end{equation}	
	
	The extension \eqref{eq:Raynaud1} is locally trivial.  Let $\mathfrak{B}_k$ be the special fiber of $\mathfrak B$.  We fix a cover by affine open subvarieties $U_\alpha \subset \mathfrak{B}_k$ and trivializations of the Raynaud extension over the corresponding affinoid subdomains $V_\alpha \subset B$.  
	
	The trivialization $q_1^{-1}(V_\alpha) \cong T_1 \times V_\alpha$ induces a trivialization $q^{-1}(V_\alpha) \cong T \times V_\alpha$, and we can pullback the coordinates $x_1, \ldots, x_n$ on $T$.  These are not globally well-defined on $E$, but $\val(x_i)$ is independent of the choice of local trivialization, giving
	\[
	\trop\colon E \to \RR^n.
	\]
	Tropicalization maps $M$, the kernel of $E \to \mathcal A^{\an}$, isomorphically onto a lattice $\Lambda \subset \RR^n$.
	
	\subsection{From polytopal decompositions to formal models}  \label{sec:Gublermodels}
	
	Recall that each affinoid analytic domain has a \emph{canonical model}, which is the formal spectrum of its ring of power bounded analytic functions.  Given a polytope $Q \subset \RR^n$ that is admissible in the sense of Definition~\ref{def:admissible}, the preimage of $Q$ under the tropicalization map $T \to \RR^n$ is an affinoid analytic domain, and its canonical model is the formal completion of an affine toric variety over the valuation ring; in the framework of \cite{Gubler13, GublerSoto15}, it is the formal completion $\mathfrak{U}_Q$ of the affine toric variety $U_Q$ corresponding to the cone over $Q \times \{1\}$ in $\RR^n \times \RR_{\geq 0}$.
	
	Now, consider an atlas $\{ U_\alpha : \alpha \in I \}$ of affine opens in $\mathfrak{B}_k$ such that the Raynaud extension is trivial over the corresponding affinoid generic fibers $V_\alpha \subset B$.  Then $$V_{Q,\alpha} = q^{-1}(V_\alpha) \cap \trop^{-1}(Q)$$ is an affinoid subdomain in $E$, and the trivialization identifies $V_{Q,\alpha}$ with the product of $V_\alpha$ with the preimage of $Q$ in $T$.
	
	Let $\mathfrak{U}_\alpha$ be the formal open subscheme of $\mathfrak{B}$ supported on $U_\alpha$.  Recall that a formal model of an affinoid domain is the canonical model if and only if its special fiber is reduced \cite[Proposition~3.13]{BPR16}.  Hence $\mathfrak{U}_\alpha$ is the canonical model of $V_\alpha$, and the canonical  model of $V_{Q, \alpha}$ is
	\[
	\mathfrak{U}_{Q, \alpha} = \mathfrak{U}_{\alpha} \times_{\Spf R} \mathfrak{U}_Q.
	\]  
	
	If $Q'$ is a face of an admissible polytope $Q$ and $U_\beta \subset U_\alpha$, then $\mathfrak{U}_{Q',\beta}$ is a formal affine open subvariety of $\mathfrak{U}_{Q,\alpha}$. 
	
	\begin{defn}
		Let $\Delta$ be an admissible polytopal decomposition of $\RR^n$.  Then the Gubler model $\mathfrak{U}_\Delta$ is the admissible formal scheme obtained from $\{\mathfrak{U}_{Q, \alpha} : Q \in \Delta \}$ by gluing along the inclusions $\mathfrak{U}_{Q', \beta} \subset \mathfrak{U}_{Q,\alpha}$ whenever $U_\beta \subset U_\alpha$ and $Q'$ is a face of $Q$.
	\end{defn}
	
	Above we used the choice of an atlas $U_\alpha$ that trivializes the Raynaud extensions. Since any two such choices have a common refinement the construction of the Gubler model does not depend on this choice.
	
	Let $U_{Q, \alpha}$ denote the special fiber of $\mathfrak{U}_{Q, \alpha}$.  Note that
	\begin{equation}\label{eq:cover}
		\left \{ U_{Q, \alpha} : Q \in \Delta \mbox{ and } \alpha \in I \right \} 
	\end{equation}
	is an affine open cover of the special fiber of $\mathfrak{U}_\Delta$.  
	
	By construction, $\mathfrak{U}_\Delta$ is a formal model of $\trop^{-1}(|\Delta|)$. In particular, if $|\Delta| = \RR^n$ then $\mathfrak{U}_\Delta$ is a formal model of $E$, which we denote $\mathfrak{E}_\Delta$.  If $\Delta$ is $\Lambda$-periodic then $M = \ker(E \to \mathcal A^{\an})$ acts naturally and freely on $\mathfrak{E}_\Delta$, and the quotient $$ \mathfrak{A}_\Delta = \mathfrak{E}_\Delta / M$$ is a formal model of $\mathcal{A}^\an$.  This gives a formal model of $\mathcal{A}^\an$ naturally associated to each admissible polytopal decomposition of its skeleton; such models are called \emph{Mumford models}. Here an admissible polytopal decomposition of the skeleton is by definition a periodic admissible decomposition of its universal cover $\RR^n$.
	If the Mumford model $\mathfrak{A}_\Delta$ is projective, then it is uniquely algebraizable, as the formal completion of a projective scheme $\mathcal{A}_\Delta$ over $\Spec R$.  We then refer to the algebraic integral model $\mathcal{A}_\Delta$ as an \emph{algebraic Mumford model} of $\mathcal A$.
	
	\begin{thm} \label{thm:Mumford}
		Let $\cX$ be a semistable curve over the valuation ring $R$ with skeleton $\Gamma$, and let $\mathcal A = \Pic^d(\cX_K)$.  Let $\Delta_{\phi}$ be the Namikawa decompositions of the skeleton $\Pic^d(\Gamma) \subset \Pic^d(\cX_K)^\an$ into $\phi$-polystable types. Then the Mumford model  $\mathfrak{A}_{\Delta_{\phi}}$ is projective and the compactified Jacobian $\overline J_\cX(\phi)$ is naturally isomorphic to the algebraic Mumford model $\mathcal {A}_{\Delta_{\phi}}$.
	\end{thm}
	
	\begin{proof}
		As in \S\ref{sec:stabledecomp}, we consider the decomposition of $\Omega(\Gamma)^*$ by polytopes corresponding to $\phi$-polystable types 
		$$\widetilde \Delta_{\phi} = \left \{ Q_{S, \ud}(\xi) : (S, \ud) \mbox{ is } \phi \mbox{-polystable} \right \}.$$
		Recall that the decomposition is periodic with respect to the lattice $\Lambda = \mu(H_1(G,\ZZ)$. 
		
		Let $\mathcal{A} = \Pic^d(\mathcal{X}_K)$, and let $A = \mathcal{A}^\an$ with uniformization $E \to A$.  As in \S\ref{sec:toriccharts}, $\widetilde \Delta_{\phi}$ gives rise to a locally finite type variety $E_{\widetilde \Delta_{\phi}}$ over the residue field, with a free $\Lambda$-action. By construction $E_{\widetilde \Delta_{\phi}}$ is the special fiber of the formal model $\mathfrak E = \mathfrak{E}_{\widetilde \Delta_{\phi}}$.  The quotient $E_{\widetilde \Delta_{\phi}}/\Lambda$ is isomorphic to the special fiber $\overline J_{X}(\phi)$ of the compactified Jacobian $\overline J_{\mathcal X}(\phi)$, by Proposition~\ref{prop:specialquotient}.  This isomorphism gives rise to a second formal scheme $\mathfrak{E}'$ whose underlying topological space is $E_{\widetilde \Delta_{\phi}}$, by pulling back the structure sheaf on the formal completion of the compactified Jacobian $\overline J_{\mathcal X}(\phi)$ along its special fiber.  We will show that these two formal schemes $\mathfrak E$ and $\mathfrak E'$ are canonically isomorphic.
		
		As in \S\ref{sec:uniformization}, we have a formal abelian scheme $\mathfrak{B}$ with generic fiber $B$ and an exact sequence of analytic groups $1 \to T \to E \to B \to 1$.  In our case, the special fiber $\mathfrak{B}_k$ is the Jacobian of the normalized special fiber $X^\nu$, by \eqref{eq:Jacnu}.  We fix an affine open cover $\{ U_\alpha : \alpha \in I \}$ of $\mathfrak B_k$ over which the Raynaud extension is trivial.  
		
		Consider the cover of $E_{\widetilde \Delta_{\phi}}$ by affine opens $\{U_{Q, \alpha} : Q \in \Delta_{\phi} \mbox{ and } \alpha \in I \}$, as in \eqref{eq:cover}. Fix $Q = Q_{S,\ud}(\xi)$ and $\alpha$, and set
		$
		U = U_{Q, \alpha}.
		$
		Note that the special fibers $\mathfrak{E}|_{U}$ and $\mathfrak{E'}|_U$ are reduced (for $\mathfrak{E'}$, this follows from 
		\cite[Theorem~4.3.7]{HuybrechtsLehn97} and \cite[Theorem 11.4(3)]{OdaSeshadri79}).  Also, since $U$ is affine, the respective generic fibers $V$ and $V'$ are affinoid.  By \cite[Proposition~3.13]{BPR16}, it follows that $\mathfrak{E}|_{U}$ and $\mathfrak{E'}|_{U}$ are the canonical models of $V$ and $V'$, respectively.  We claim that $V = V'$.
		
		Recall that $V$ is characterized in terms of uniformization and tropicalization by
		$$
		V = q^{-1}(V_\alpha) \cap \trop^{-1}(Q).
		$$
		Note that $Q = Q_{S, \ud}(\xi)$ is the disjoint union of the relative interiors of its faces.
		Therefore $V$ decomposes as a disjoint union $$V = \bigsqcup_{Q' \preceq Q} (q^{-1}(V_\alpha) \cap \trop^{-1}((Q')^\circ).$$ Each face $Q'$ is of the form $Q' = Q_{S', \ud'}(\xi')$, and the relative interior $(Q')^\circ$ maps bijectively onto its image in $\Pic^d(\Gamma)$.  By construction (and with the path $\xi'$ fixed), the points of $\trop^{-1}((Q')^\circ)$ correspond naturally and bijectively to isomorphism classes of line bundles $\mathcal{O}(D)$ on $\cX_K$ such that $\trop(D)$ is equivalent to a divisor of type $(S', \ud')$ on $\Gamma$.
		
		Similarly, $V'$ decomposes as a disjoint union over faces $Q' = Q_{S', \ud'}(\xi')$ of $Q$
		\[
		V' = \bigsqcup_{Q' \preceq Q} V^\circ_{Q'} 
		\]
		where $V^\circ_{Q'}$ is identified with isomorphism classes of line bundles on $\cX_K$ that are in $q^{-1}(V_\alpha)$ and extend to torsion free sheaves of type $(S',\ud')$ on $\cX$.  Now Proposition~\ref{prop:extension-tf} tells us that a line bundle $\mathcal{O}(D)$ on $\cX_K$ extends to a torsion free sheaf of type $(S', \ud')$ on $\cX$ if and only if $\trop(D)$ is equivalent to a divisor of type $(S', \ud')$.  It follows that $V^\circ_{Q'} = q^{-1}(V_\alpha) \cap \trop^{-1}((Q')^\circ)$, and hence $V = V'$, as claimed.
		
		We have shown, for each open set $U$ in our atlas $\{U_{Q, \alpha} \}$, that $\mathfrak{E}|_U$ and $\mathfrak{E'}|_U$ are the canonical models of the same affinoid $V = V'$, and hence they are canonically isomorphic.  Gluing these isomorphisms over all open sets in the atlas gives a canonical isomorphism $\mathfrak E \cong \mathfrak E'$.
		This isomorphism is $\Lambda$-equivariant and descends to an isomorphism between the formal Mumford model $\mathfrak{A}_{\Delta_{\phi}} = \mathfrak E/ \Lambda$ and $\mathfrak{E'}/\Lambda$, which is the formal completion of the compactified Jacobian $\overline J_{\mathcal X}(\phi)$ along its special fiber.  Finally, since $\overline J_{\mathcal X}(\phi)$ is projective, this formal completion is uniquely algebraizable, and we conclude that the isomorphism of formal models is induced by an isomorphism of algebraic integral models $\mathcal{A}_{\Delta_{\phi}} \cong \overline J_{\mathcal X}(\phi)$.	
	\end{proof}
	\bibliographystyle{amsalpha}
	\bibliography{math}
\end{document}